\documentclass{clsrbt}

\usepackage{blkarray}
\usepackage{array}
\usepackage{mathrsfs}  
\usepackage{amsmath}
\usepackage{amssymb}
\usepackage{multirow}
\usepackage[mathscr]{euscript}

\usepackage{mathdots}

\usepackage[all]{xy}
\newcommand{\xyR}[1]{\xydef@\xymatrixrowsep@{#1}}
\newcommand{\xyC}[1]{\xydef@\xymatrixcolsep@{#1}}
\newtheorem{theorem}{Theorem}[section]
\newtheorem{proposition}[theorem]{Proposition}
\newtheorem{lemma}[theorem]{Lemma}
\newtheorem{corollary}[theorem]{Corollary}
\newnumbered{definition}[theorem]{Definition} 
\newnumbered{conjecture}[theorem]{Conjecture}
\newnumbered{example}[theorem]{Example}
\newnumbered{remark}[theorem]{Remark}
\newnumbered{assumption}[theorem]{Assumption}
\newnumbered{notation}[theorem]{Notation}
\newnumbered{construction}[theorem]{Construction}

\usepackage[all]{xy}
\usepackage{url}

\def\rad{\mathop{\mathrm{rad}}}

\def\im{\mathop{\mathrm{im}}}

\title{$\Sigma$-pure-injectives for gentle algebras} 

\classno{16G20 (primary)}

\begin{document}
\title[$\Sigma$-pure-injectives in homotopy categories.]{$\Sigma$-pure-injectives in homotopy categories for gentle algebras.}
\author{Raphael Bennett-Tennenhaus.}

\maketitle
\begin{abstract}
We consider the homotopy category of complexes of projective modules over any gentle algebra. We prove that indecomposable $\Sigma$-pure-injective objects in s must be shifts of string or band complexes. We begin with a survey of purity in compactly generated triangulated categories, recalling some characterisations of $\Sigma$-pure-injective objects that mimic classical results from the model theory of modules. We then specify to the aforementioned homotopy category, and describe the compact objects in this case. Our proof uses a recent adaptation of a classification technique known as the functorial filtrations method. One of the key steps in our employment of this method is an interpretation of the appropriate linear relations in terms of pp formulas in a canonical multi-sorted language.
\end{abstract}
\section{Introduction.}\label{intro}
A fundamental idea in model theory is to study structures via the formulas they satisfy. By the famous quantifier elimination result of Baur \cite{Bau1976}, any such formula in the language of modules is equivalent to a Boolean combination of positive-primitive (pp) formulas. Module embeddings which reflect solutions to pp formulas with constants from the domain, so-called pure embeddings, became of particular interest. Attention was also placed on modules which are pure-injective, that is, injective with respect to pure embeddings. 

To study the model theory of modules, Ziegler \cite{Zie1984} defined a topological space whose points are indecomposable pure-injective modules. The introduction of what is now known as the Ziegler spectrum motivated the provision of examples. By using ideas from representation theory, Ringel \cite{Rin1995} gave a combinatorial way of constructing pure-injective modules for string algebras. In particular, just as indecomposable finite-dimensional modules for string algebras may be described using words in the sense of Butler and Ringel \cite{ButRin1987}, infinite versions of these words describe various points of the Ziegler spectrum. 

Prest and Puninski \cite{PrePun2016} then verified a conjecture from \cite{Rin1995}, which says that, for domestic string algebras, the constructions from \cite{Rin1995} give an exhaustive list of indecomposable pure-injective modules. As part of the proof, these authors studied coherent functors from the category of pp-pairs, defined by evaluation. Infinite words have also been used for module classifications by Crawley-Boevey \cite{Cra2018}, who considered string algebras which need not be domestic, and need not be finite-dimensional. More recently, in joint work \cite{BenCra2018} with Crawley-Boevey, we adpted the focus of \cite{Cra2018} to classify $\Sigma$-pure-injective modules: those for which any small direct sum copies of itself remains pure-injective.

Two key ideas were used in \cite{BenCra2018}. The first of these, employed heavilly in \cite{PrePun2016}, was to translate the combinatorial properties defining string algebras into properties of pp-formulas, and then use results from model-theoretic algebra. The second is a classification technique coming from representation theory known as the functorial filtrations method, and goes back to work of Gelfand and Ponomarev \cite{GelPon1968}. This method involves verifying a certain splitting result about linear relations, which was verified for $\Sigma$-pure-injective relations in \cite{BenCra2018}.

In this article we mimic some of the ideas and results above, but in a different categorical context. In particular, we replace module-theoretic notions with their counterparts in the setting of compactly generated triangulated categories. The notion of compactness we refer to originates from algebraic topology, and the axioms were generalised by Neeman \cite{Nee1992} to all triangulated categories. Since then such categories have become of consideredable interest in algebra. From the perspective of model-theoretic algebra, Krause \cite{Kra2002} provided the definitions of pure monomorphisms and pure-injective objects, and Garkusha and Prest \cite{GarPre2005} subsequently introduced a multi-sorted language. 

The definitions and results from  \cite{Kra2002} and \cite{GarPre2005} are analogous to ideas from the model theory of modules. In particular, just as any module is a structure for the language of modules as in  \cite{Bau1976}, objects in compactly generated triangulated categories give rise to structures in the aforementioned multi-sorted language. In more recent work \cite{Ben2020}, the author provided some characterisations of $\Sigma$-pure-injective objects.

By a result of Neeman \cite{Nee2008}, the homotopy category of unbounded complexes of projective modules over any coherent ring is compactly generated. Assem and Skowro{\'n}ski \cite{AssSko1987} studied iterated-tilted algebras of type $\mathbb{A}$, and showed these algebras may be presented using a bounded quiver with relations satisfying various combinatorial properties. The resulting class of finite-dimensional algebras they define are known as gentle algebras. Since finite-dimensional algebras are coherent, the aforementioned homotopy category for these algebras is compactly generated.

By a result of Vossieck \cite{Voss2001}, so-called dervied-discrete algebras are derived equivalent to path algebras of $\mathbb{A}\text{-}\mathbb{D}\text{-}\mathbb{E}$ Dynkin quivers, or gentle algebras of a particular form. Arnesen, Laking, Pauksztello and Prest \cite{ArnLakPauPre2017} studied the latter, and classified pure-injective indecomposable objects in the homotopy category discussed above.

In the author's thesis \cite{Ben2018}, the definition of a gentle algebra from \cite{AssSko1987} was generalised to include examples of infinite-dimensional algebras, and even rings of mixed representation type. Here we called the rings complete gentle algebras, and objects were classified in the full subcategory of complexes with finitely generated homogeneous components. Our main result in this article, Theorem \ref{maincor}, can been seen as an adaptation of the classifications achieved: in \cite{BenCra2018}, where one restricts string algebras to finite-dimensional gentle algebras (as in \cite{AssSko1987}), and extends the module category to the homotopy category; in \cite{ArnLakPauPre2017}, where one extends (non-Dynkin) derived-discrete algebras to gentle algebras, and restricts the focus from pure-injective objects to those which are $\Sigma$-pure-injective; and in \cite{Ben2018}, where one restricts complete gentle algebras to finite-dimensional ones, and replaces complexes with finitely generated homogeneous components with those defining $\Sigma$-pure-injective objects in the homotopy category. 
\begin{theorem}\label{maincor} Let $k$ be a field, $k[T,T^{-1}]$ the ring of Laurent polynomials, and $\Lambda$ be a gentle algebra over $k$. Every $\Sigma$-pure-injective object in the homotopy category $\mathcal{K}(\Lambda\text{-}\bf{Proj})$ of complexes of projective modules is isomorphic to a coproduct of (shifts of) string complexes and band complexes paramenterised by $\Sigma$-pure-injective $k[T,T^{-1}]$-modules. 
\end{theorem}
The article is organised as follows. In \S\ref{1} we recall the notion of a compactly generated triangulated category and recall some charaterisations of $\Sigma$-pure-injective objects in this setting.  In \S\ref{2} we recall the definition of a gentle algebra and the definitions of string and band complexes. In \S\ref{cmpctgentle} we describe compact indecomposable objects in the homotopy categoy. In \S\ref{linrels} we recall the category of linear relations, survey results about pure-injective relations and show how so-called homotopy words define such relations. In \S\ref{seccompact} we introduce the notion of a compact morphism associated to a finite homotopy word. In \S\ref{splitsection} we study the relationship between $\Sigma$-pure-injective complexes and $\Sigma$-pure-injective relations. In \S\ref{6} we set up the prerequisites needed to employ the functorial filtrations method. In \S\ref{seccover} we verify some compatbility conditions required to use this method. In \S\ref{sec:Proofs-of-the} we give a proof of Theorem \ref{maincor}.
\section{Compactly generated triangulated categories.}\label{1}
\begin{definition}\label{compactob}\cite{Nee1992,Nee1996} Let $\mathcal{T}$ be a skeletally small triangulated category, with suspension functor $\Sigma$, and in which all small coproducts exist. An object $X$ of $\mathcal{T}$ is said to be \textit{compact} if, for any set $I$ and collection $\mathrm{M}=\{M_{i}\mid i\in I\}$ of objects in $\mathcal{T}$, the morphism $\bigoplus_{i}\mathcal{T}(X,M_{i})\to\mathcal{T}(X,\bigoplus_{i}M_{i})$, given by the universal property of the coproduct, is an isomorphism. Let $\mathcal{T}^{c}$ be the full triangulated subcategory of $\mathcal{T}$ consisting of compact objects.

Given a set $\mathcal{G}$ of compact objects in $\mathcal{T}$, we say that $\mathcal{T}$ is \textit{compactly generated by} $\mathcal{G}$ if there  are no non-zero objects $M$ in $\mathcal{T}$ satisfying $\mathcal{T}(X,M)=0$ for all $X\in\mathcal{G}$ (or, said another way, any non-zero object $M$ gives rise to a non-zero morphism $X\to M$ for some $X\in\mathcal{G}$).

Let $\Sigma$ denote the suspension functor associated to the triangulated structure of $\mathcal{T}$. If $\mathcal{T}$ is compactly generated by $\mathcal{G}$ we call $\mathcal{G}$ a \textit{generating set} provided $\Sigma X\in \mathcal{G}$ for all $X\in\mathcal{G}$. 
\end{definition}
As discussed in \S\ref{intro}, the examples of compactly generated triangulated categories we focus on will be homotopy categories for gentle algebras. For this we fix notation for the sequel. 
\begin{notation}Let $\Lambda$ be a unital ring. In the remainder of the article write $\mathcal{C}(\Lambda\text{-}\mathbf{Proj})$ for the category of complexes $X$ of projective $\Lambda$-modules, denoted
\[
X=\xymatrix{
\cdots\ar[r]^{d_{X}^{i-2}} & X^{i-1}\ar[r]^{d_{X}^{i-1}} & X^{i}\ar[r]^{d_{X}^{i}} & X^{i+1}\ar[r]^{d_{X}^{i+1}} & \cdots}
\]
As usual, we consider various full subcategories of $\mathcal{C}(\Lambda\text{-}\mathbf{Proj})$, defined and denoted as follows. A complex $X$ in $\mathcal{C}(\Lambda\text{-}\mathbf{Proj})$ lies in:
\begin{enumerate}
\item $\mathcal{C}(\Lambda\text{-}\mathbf{proj})$ if and only if each $X^{i}$ is a finitely generated (projective) $\Lambda$-module; 
\item the category $\mathcal{C}^{+}(\Lambda\text{-}\mathbf{proj})$ of \textit{left}-\textit{bounded} complexes if and only if $X$ lies in $\mathcal{C}(\Lambda\text{-}\mathbf{proj})$ and there exists $n_{X}\in\mathbb{Z}$ with $X^{i}=0$ whenever $i< n_{X}$;
\item the category $\mathcal{C}^{-}(\Lambda\text{-}\mathbf{proj})$ of \textit{right}-\textit{bounded} complexes if and only if  $X$ lies in $\mathcal{C}(\Lambda\text{-}\mathbf{proj})$ and there exists $m_{X}\in\mathbb{Z}$ with $X^{i}=0$ whenever $i>m_{X}$;
\item the subcategory $\mathcal{C}^{\pm,b}(\Lambda\text{-}\mathbf{proj})$ of $\mathcal{C}^{\pm}(\Lambda\text{-}\mathbf{proj})$ consisting of complexes with \textit{bounded cohomology} if and only if ($X$ lies in $\mathcal{C}^{\pm}(\Lambda\text{-}\mathbf{proj})$ and) there exists $h_{X}\in\mathbb{Z}$ with $\mathrm{im}(d_{X}^{i})=\mathrm{ker}(d_{X}^{i+1})$ whenever $i\notin[-h_{X},h_{X}]$.
\end{enumerate}
We denote the homotopy category associated to $\mathcal{C}(\Lambda\text{-}\mathbf{Proj})$ (respectively $\mathcal{C}(\Lambda\text{-}\mathbf{proj})$, respectively $\mathcal{C}^{\pm}(\Lambda\text{-}\mathbf{proj})$, respectively $\mathcal{C}^{\pm,b}(\Lambda\text{-}\mathbf{proj})$) by $\mathcal{K}(\Lambda\text{-}\mathbf{Proj})$ (respectively $\mathcal{K}(\Lambda\text{-}\mathbf{proj})$, respectively $\mathcal{K}^{\pm}(\Lambda\text{-}\mathbf{proj})$, respectively $\mathcal{K}^{\pm,b}(\Lambda\text{-}\mathbf{proj})$).
\end{notation}

Theorem \ref{neemantheorem} summarises some details of a generalisation, due to Neeman \cite[Theorem]{Nee2008}, of a result due to J\o rgenson \cite[Theorem 3.2]{Jor2005}.

\begin{theorem}\label{neemantheorem}\emph{\cite[Proposition 7.12]{Nee2008}} Let $\Lambda$ be a unital ring. Then the full subcategory $\mathcal{K}(\Lambda\text{-}\mathbf{Proj})^{c}$ of $\mathcal{K}(\Lambda\text{-}\mathbf{Proj})$ consisting of compact objects is precisely $\mathcal{K}^{+,b}(\Lambda\text{-}\mathbf{proj})$.

\emph{\cite[Theorem 1.1(i)]{Nee2008}} If, also, $\Lambda$ is right coherent, then $\mathcal{K}(\Lambda\text{-}\mathbf{Proj})$ is compactly generated.
\end{theorem}
\begin{assumption}\label{ass22}In the remainder of \S\ref{1} fix a triangulated category $\mathcal{T}$ with suspension functor $\Sigma$, and we assume that $\mathcal{T}$ is skeletally small,
that $\mathcal{T}$ has all small coproducts, and that $\mathcal{T}$ is compactly generated. Additionally, we fix a set $\mathcal{S}$ of objects in $\mathcal{T}^{c}$ consisting of exactly one representative for each isoclass.
\end{assumption}
\begin{definition}\label{ppdefsub}\cite[\S 2]{GarPre2005} Let $M$ be an object of $\mathcal{T}$. For any morphism $a:X\to Y$ in $\mathcal{T}^{c}$  let $\mathcal{T}(a,M):\mathcal{T}(Y, M)\to \mathcal{T}(X, M)$ be the morphism given by $f\mapsto fa$, and let
\[
Ma=\mathrm{im}(\mathcal{T}(a,M))=\{fa\in \mathcal{T}(X, M)\mid f\in\mathcal{T}(Y, M)\}
\]
For any compact object $X$, by a \textit{pp}-\textit{definable subgroup of} $M$ \textit{of sort} $X$ we mean a subgroup of $\mathcal{T}(X,M)$ of the form $Ma$ for some (object $Y$ of $\mathcal{T}^{c}$ and some morphism) $a$ as above. 
\end{definition}
In Corollary \ref{compactcor} we directly combine  Theorem \ref{neemantheorem} and Definition \ref{ppdefsub}.
\begin{definition}\cite[Definition 1.1]{Kra2002} A morphism $h:L\to M$ in $\mathcal{T}$ is called a \textit{pure monomorphism} if $\mathbf{Y}(h)_{X}:\mathcal{T}(X,L)\to\mathcal{T}(X,M)$ is injective for each object $X$ of $\mathcal{T}^{c}$. An object $M$ of $\mathcal{T}$ is called \textit{pure}-\textit{injective} if each pure monomorphism $M\to N$ is a section, and $M$ is called $\Sigma$-\textit{pure}-\textit{injective} if, for any set $I$, the coproduct $\bigoplus_{i}M$ is pure-injective.
\end{definition}
\begin{remark}\label{brownproducts}
As a result of Assumption \ref{ass22}, by the \textit{Brown representability theorem} we have that $\mathcal{T}$ has all small products. See \cite[Lemma 1.5]{Kra2000} for details.
\end{remark}
The full statement of Theorem \ref{characterisation}, namely \cite[Theorem 1.1]{Ben2020}, is analogous to the summary in a book of Jensen and Lenzing  \cite[Theorem 8.1]{JenLen1989} of various well-known characterisations of $\Sigma$-pure-injective modules. 
\begin{theorem}\label{characterisation}\emph{\cite[Theorem 1.1(i,iv,v,vi)]{Ben2020}}
Let $M$ be an object of $\mathcal{T}$. Then the following statements are equivalent.
\begin{enumerate}
\item $M$ is $\Sigma$-pure injective.
\item For any set $I$ the canonical morphism from $M^{(I)}$ to the product $M^{I}$ is a section.
\item For any object $X$ of $\mathcal{T}^{c}$ each descending chain $\varphi_{1}(M)\supseteq \varphi_{2}(M)\supseteq\dots$  of pp-definable subgroups of $M$ of sort $X$ must eventually stabilise.
\item $M$ is pure injective, and for any set $I$ the object $M^{I}$ is isomorphic to a coproduct of indecomposable pure-injective objects with local endomorphism rings.
\end{enumerate}
\end{theorem}
The characterisations \cite[Theorem 1.1(i,ii,iii)]{Ben2020}, which were omitted in Theorem \ref{characterisation}, are not used in the sequel. To state Corollary \ref{sigmacardy} we need to recall a notion of cardinality. As discussed in \S\ref{intro}, Garkusha and Prest \cite{GarPre2005} introduced a multi-sorted \textit{canonical language} associated to the category $\mathcal{T}$ using the restricted of Yoneda functor. Objects in $\mathcal{T}$ give rise to structures in this language, and Definition \ref{defcardy} is the specification of the notion of the cardinality of a multi-sorted structure to the setting of compactly generated triangulated categories.
\begin{definition}\label{defcardy}
Let $M$ be an object in $\mathcal{T}$. The \textit{cardinality of the structure} $\mathsf{M}$ \textit{underlying} $M$ is defined and denoted $\vert\mathsf{M}\vert=\vert\bigsqcup_{G\in\mathcal{S}}\mathcal{T}(G,M)\vert$.
\end{definition}
We note two more results from \cite{Ben2020} to be employed later; see Lemma \ref{cmpctgen} and Remark \ref{lastremark}.
\begin{corollary}\label{sigmacardy}\emph{\cite[Corollary 7.1]{Ben2020}} There exists a cardinal $\kappa$ such that the cardinality of the structure underlying any indecomposable pure-injective object of $\mathcal{T}$ is at most $\kappa$.
\end{corollary}
\begin{corollary}\emph{(}See \emph{\cite[Corollary 7.6]{Ben2020}).}\label{decompcorr} Let $\mathrm{L}=\{L_{j}\mid j\in J\}$ and $\mathrm{M}=\{M_{k}\mid j\in J\}$ be collections of indecomposable objects in $\mathcal{T}$ such that the coproducts $\bigoplus _{j\in J}L_{j}$ and $\bigoplus _{k\in K}M_{k}$ are pure injective, and such that $\bigoplus _{j\in J}L_{j}\simeq \bigoplus _{k\in K}M_{k}$.  Then there is a bijection $\sigma:J\to K$ with $L_{j}\simeq M_{\sigma(j)}$ for all $j\in J$.
\end{corollary}
\section{Gentle algebras and string and band complexes.}\label{2}
In the remainder of the article we use notation both from the author's PhD thesis \cite{Ben2018} and from work of Bekkert and Merklen \cite{BekMer2003}. We will also refer to a summary \cite{Ben2016} of a part of \cite{Ben2018}.
\begin{notation}For the remainder of the article we let $k$ be a field, $Q$ be a quiver and $kQ$ be the path algebra. For any vertex $u$ we let $\mathbf{A}(u\rightarrow)$ (respectively $\mathbf{A}(\rightarrow u)$) denote the set of arrows $a$ whose tail $t(a)$ (respectively head $h(a)$) is $u$.  Our conventions on multiplication in $kQ$ mean that, if $a\in\mathbf{A}(\rightarrow v)$ and $b\in\mathbf{A}(u\rightarrow)$, then $ba$ if $u=v$ and $ba=0$ otherwise. 

\cite[Definition 1.1.13]{Ben2018} Any non-trivial path $\gamma$ in $Q$ has a \textit{first arrow} $\mathrm{f}(\gamma)$ and a \textit{last arrow} $\mathrm{l}(\gamma)$ satisfying $\mathrm{l}(\gamma)\gamma'=\gamma=\gamma''\mathrm{f}(p)$
for some (possibly trivial) paths $\gamma'$ and $\gamma''$. The head and tail of $\gamma$ are defined and denoted respectively by $h(\gamma)=h(\mathrm{l}(\gamma))$ and $t(\gamma)=t(\mathrm{f}(\gamma))$.
\end{notation}
\begin{assumption}
For the remainder of the article we assume $\Lambda$ is a \textit{gentle algebra} in the sense of \cite{AssSko1987}. That is, we assume that $\Lambda=kQ/\mathcal{J}$ where $\mathcal{J}$ is an admissible ideal in $kQ$ generated by a set of length $2$ paths, such that the following conditions hold.
\begin{enumerate}
\item If $v$ is a vertex then $\vert\mathbf{A}(v\rightarrow)\vert\leq2$
and $\vert\mathbf{A}(\rightarrow v)\vert\leq2$.
\item If $y\in\mathbf{A}$ then $\vert\{ x\in \mathbf{A}(h(y)\rightarrow)\mid xy\in\mathcal{J}\}\vert\leq 1$ and $\vert \{ z\in \mathbf{A}(t(y)\rightarrow)\mid yz\in\mathcal{J}\}\vert\leq 1$.
\item If $y\in\mathbf{A}$ then $\vert \{ x\in \mathbf{A}(h(y)\rightarrow)\mid xy\notin\mathcal{J}\}\vert\leq 1$ and $\vert\{ z\in \mathbf{A}(t(y)\rightarrow)\mid yz\notin\mathcal{J}\}\vert\leq 1$.
\end{enumerate}
We denote by $\mathbf{P}$ the set of non-trivial paths $\gamma\notin \mathcal{J}$.
\end{assumption}
Note that, since $\mathcal{J}$ is assumed to be generated by length 2 paths, we must have $\mathbf{A}\subseteq\mathbf{P}$. We now choose an example of a gentle algebra to be repeatedly recycled in the sequel.
\begin{example}\label{runningexample}Let $Q$ be the quiver 
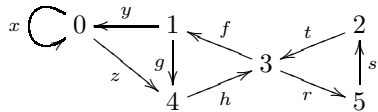
\[
\xymatrix@R=.2em{0\ar@(ul,dl)_{x}\ar[ddr]_{z} & 1\ar[l]_{y}\ar[dd]_{g} &  & 2\ar[dl]_{t}\\
& & 3\ar[ul]_{f}\ar[dr]_{r} &\\
 & 4\ar[ur]_{h} &  & 5\ar[uu]_{s}
}
\]
Then, letting  $\mathcal{J}=\langle x^{2},\,zy,\,gf,\,hg,\,fh,\,sr,\,ts,\,rt\rangle$, the ring $\Lambda=kQ/\mathcal{J}$ is a gentle algebra. The paths in $\mathbf{P}$ of length $2$ are $rh$, $hz$, $zx$, $xy$, $yf$ and $ft$. The longer paths in $\mathbf{P}$ may be found by amalgamating paths at common arrows. For example, doing this with $xy$ and $yf$ gives $xyf$. Note that $\Lambda$ is indeed finite-dimensional, since any path in $\mathbf{P}\setminus\mathbf{A}$ is a subpath of $\gamma=rhzxyft$. The first arrow of $\gamma$ is $\mathrm{f}(\gamma)=t$, and the last is $\mathrm{l}(\gamma)=r$.  
\end{example}
We now recall the language of \textit{generalised words}, in the sense of Bekkert and Merklen \cite{BekMer2003}, in order to define string and band complexes; see Definition \ref{stringcomplexes}. In Definition \ref{homotopywords} we recall the alternative system of \textit{homotopy words} from \cite{Ben2018}. It is straightforward to translate between the two word systems. The author finds it easier to define string and band complexes using generalised words, although homotopy words are necessary for the application of the functorial filtrations method (as discussed in \S\ref{intro}).
\begin{definition}\label{definition2}\cite[\S 4.1]{BekMer2003} By a \textit{generalised letter} we mean a symbol of the form $[\gamma]$ or $[\gamma^{-1}]$ where $\gamma\in\mathbf{P}$. Let $I$ be one of the
sets $\{0,\dots,m\}$ (for some $m\geq0$), $\mathbb{N}$, $-\mathbb{N}=\{-n\mid n\in\mathbb{N}\}$,
or $\mathbb{Z}$. For $I\neq\{0\}$ By a \textit{generalised} $I$-\textit{word} we mean a sequence of the form 
\[
[C]=\begin{cases}
[C_{1}]\dots [C_{m}] & (\mbox{if }I=\{0,\dots,m\})\\
[C_{1}][C_{2}]\dots & (\mbox{if }I=\mathbb{N})\\
\dots[C_{-2}][C_{-1}][C_{0}] & (\mbox{if }I=-\mathbb{N})\\
\dots\dots[C_{-2}][C_{-1}][C_{0}]:[C_{1}][C_{2}]\dots & (\mbox{if }I=\mathbb{Z})
\end{cases}
\](which will be written as $[C]=\dots [C_{i}]\dots$ to save
space) where each $[C_{i}]$ is a generalised letter, and any sequence of the form $[C_{i}][C_{i+1}]$ is one of 
\begin{enumerate}
\item $[\gamma][\lambda^{-1}]$ where $h(\gamma)=h(\lambda)$ and $\mathrm{l}(\gamma)\neq \mathrm{l}(\lambda)$;
\item $[\gamma^{-1}][\lambda^{-1}]$ where $t(\gamma)=h(\lambda)$ and $\mathrm{f}(\gamma)\mathrm{l}(\lambda)\in\mathcal{J}$;
\item $[\gamma^{-1}][\lambda]$ where $t(\gamma)=t(\lambda)$ and $\mathrm{f}(\gamma)\neq \mathrm{f}(\lambda)$;
\item $[\gamma][\lambda]$ where $h(\gamma)=t(\lambda)$ and $\mathrm{f}(\lambda)\mathrm{l}(\gamma)\in\mathcal{J}$.
\end{enumerate}
For $I=\{0\}$ there are \textit{trivial} generalised words
$ [1 _{v,1}]$ and $ [1 _{v,-1}]$ for each vertex
$v$. 

The head and tail of generalised letter are defined by setting $h([\gamma])=t(\gamma)$ and $t([\gamma])=h(\gamma)$. For each $i\in I$ there is an \textit{associated vertex} $v_{C}(i)$
defined by $v_{C}(i)=t([C_{i}])$ for $i\leq0$ and $v_{C}(i)=t(r_{i})$
for $i>0$ provided $I\neq\{0\}$, and $v_{ 1 _{v,\pm1}}(0)=v$ otherwise.
\end{definition}
The use of square brackets helps the reader distinguish between (sequences of consecutive arrows which make up a path in $\mathbf{P}$) and (generalised letters which are all direct or all inverse). 
\begin{example}For the gentle algebra $\Lambda$ from Example \ref{runningexample} we can generate various examples of generalised words. In what follows we will use a variety of shorthand notation. For example, we define the generalised $-\mathbb{N}$-word ${}^{\infty}([r][s][t])$ and the generalised $\mathbb{N}$-word $([x])^{\infty}$ by $=\dots[r][s][t][r][s][t][r][s][t]$ and $([x])^{\infty}=[x][x][x]\dots$ respectively. The colon helps one index generalised $\mathbb{Z}$-words in the sequel. For example, the $\mathbb{Z}$-words ${}^{\infty}([r][s][t]) : ([h^{-1}][g^{-1}][f^{-1}])^{\infty}$ and ${}^{\infty}([r][s][t]) :[r][s][t] ([h^{-1}][g^{-1}][f^{-1}])^{\infty}$
are distinct. Note that the string complexes defined in Definition \ref{stringcomplexes} by the former is a degree 3 shift of the complex defined by the latter. 
\end{example}
To apply the functorial filtrations method, as described in \S\ref{intro}, we will also introduce an alternative language of \textit{homotopy words} developed in the author's thesis \cite{Ben2018}; see Definition \ref{homotopywords}. 
\begin{definition}\label{stringcomplexes}\cite[Definition 2]{BekMer2003} Now let $I$ be one of $\{0,\dots,m\}$, $\mathbb{N}$, $-\mathbb{N}$ or $\mathbb{Z}$ as above, and fix a generalised $I$-word $[C]$. We now recall a function $\mu_{C}:I\to\mathbb{Z}$ which, in what follows below, keeps track of the homogeneous degree in a string complex. 

If $\gamma\in\mathbf{P}$ let $H[\gamma]=-1$
and $H[\gamma^{-1}]=1$.  Let $\mu_{C}(0)=0$, $\mu_{C}(i)=H[C_{1}]+\dots+H[C_{i}]$ if $0<i\in I$; and $\mu_{C}(i)=-(H[C_{0}]+\dots+H[C_{i+1}]$ if $0>i\in I$. We now define a complex $P(C)$.

 For $n\in\mathbb{Z}$ let $P^{n}(C)$ be the sum $\bigoplus\Lambda e_{v_{C}(i)}$ over $i\in \mu_{C}^{-1}(n)$. For each $i\in I$ let $ b _{i,C}$ denote the coset of $e_{v_{C}(i)}$
in $P(C)$ in degree $\mu_{C}(i)$. Define the complex $P(C)$ by extending the assignment $d_{P(C)}( b _{i})= b _{i}^{-}+ b _{i}^{+}$
linearly over $\Lambda$ for each $i\in I$, where:
\begin{enumerate}
\item $b _{i}^{+}=\alpha b _{i+1}$ if $i+1\in I$ and $[C_{i+1}]=[\alpha^{-1}]$, and $b _{i}^{+}=0$ otherwise; and
\item $b _{i}^{-}=\beta b _{i-1}$ if $i-1\in I$ and $[C_{i}]=[\beta]$, and $b _{i}^{-}=0$ otherwise.
\end{enumerate}
\end{definition}
\begin{example}\label{runningexample2}Consider the gentle algebra $\Lambda$ from Example \ref{runningexample} and the generalised word $[C]=[h^{-1}][g^{-1}][(ft)^{-1}][s^{-1}][r^{-1}][yf] [x^{-1}]$. We may depict $P(C)$ by
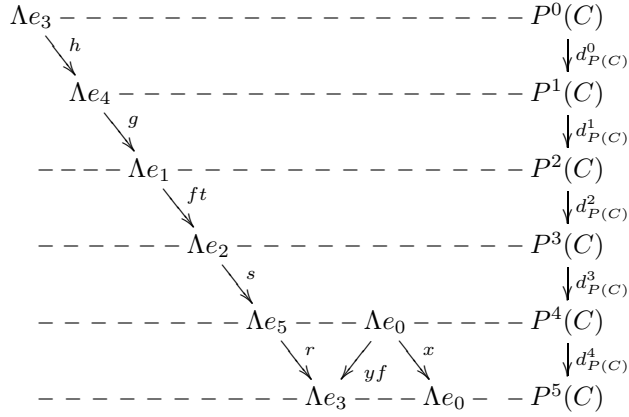
\[
\xymatrix@C=.005em@R=1.2em{ 
\Lambda e_{3}\ar[dr]^{h} & & & & & & & & & &  & P^{0}(C)\ar[d]^{d_{P(C)}^{0}}\ar@{--}[lllllllllll]\\
& \Lambda e_{4}\ar@{--}[l]\ar[dr]^{g} & & & & & & & & & & P^{1}(C)\ar[d]^{d_{P(C)}^{1}}\ar@{--}[llllllllll]\\
& & \Lambda e_{1}\ar@{--}[ll]\ar[dr]^{ft} & & & & & & & & & P^{2}(C)\ar[d]^{d_{P(C)}^{2}}\ar@{--}[lllllllll]\\
& & & \Lambda e_{2}\ar@{--}[lll]\ar[dr]^{s} & & & & & & & & P^{3}(C)\ar[d]^{d_{P(C)}^{3}}\ar@{--}[llllllll]\\
& & & & \Lambda e_{5}\ar@{--}[llll]\ar[dr]^{r} &  & \Lambda e_{0}\ar@{--}[ll]\ar[dl]^{yf}\ar[dr]^{x} & & & & & P^{4}(C)\ar[d]^{d_{P(C)}^{4}}\ar@{--}[lllll]\\
& & & & & \Lambda e_{3}\ar@{--}[lllll] & & \Lambda e_{0}\ar@{--}[ll]  & & & & P^{5}(C)\ar@{--}[llll]
}\]
where an arrow $\Lambda e_{v}\to \Lambda e_{u}$ labelled by a path $\gamma$ with head $v$ and tail $u$ indicates right-multiplication by $\gamma$. The dashed arrows are for the purposes of aiding the reader calculate the homogeneous components: for example, $P^{4}(C)=\Lambda e_{5}\oplus\Lambda e_{0}$.
\end{example}
In the sense of Example \ref{runningexample2}, the schema defining the string complex associated to a generalised $\mathbb{N}$-word (respectively $-\mathbb{N}$-word, respectively $\mathbb{Z}$-word) extends infinitely to the right (respectively to the right, respectively to the left and to the right). 

In some cases the shape of the schema for a generalised $\mathbb{Z}$-word repeats itself, in such a way that the string complex is concentrated in finitely many homogeneous components. When this happens, one may identify points in the same translational orbit. Algebraically this corresponds to defining a \textit{band complex} from a \textit{period} generalised $\mathbb{Z}$-word. We explain this below.
\begin{definition}\cite[Definitions 1.3.26, 1.3.32 and 1.3.42]{Ben2018}\label{def.3.3}
Let $[C]$ be a generalised word. Write $I_{C}$ for the subset of $\mathbb{Z}$ where $[C]$ is a generalised $I_{C}$-word. Let $t\in\mathbb{Z}$. If $I_{C}=\mathbb{Z}$ we let $[C[t]]=\dots [C_{t}]\mid[C_{t+1}]\dots$. That is, $[C_{i}[t]]_{i}=[C_{i+t}]$ for all $i\in\mathbb{Z}$. If instead $I_{C}\neq\mathbb{Z}$ we let $[C]=[C[t]]$. 

The \textit{ inverse} $[C^{-1}]$ of $[C]$ is defined
by $[1 _{v,\delta}^{-1}]= [1 _{v,-\delta}]$ if $I=\{0\}$,
and otherwise inverting the generalised letters and reversing their order. Note the generalised $\mathbb{Z}$-words are indexed so
that
\[
\left(\dots[C_{-1}][C_{0}]: [C_{1}][C_{2}]\dots\right)^{-1}=\dots [C_{2}^{-1}][C_{1}^{-1}]: [C^{-1}_{0}][C^{-1}_{-1}]\dots
\]
\cite[Definition 1.3.42]{Ben2018} We say $[C]$ is \textit{periodic}
if $I_{C}=\mathbb{Z}$, $[C]=[C\left[p\right]]$ and $\mu_{C}(p)=0$ for some $p>0$.  In this case the minimal such $p$ is the \textit{period} of $[C]$, and we say $[C]$ is $p$-\textit{periodic}. We say $[C]$ is \textit{aperiodic} if $[C]$ is not periodic. 

\cite[Definition 1.3.45]{Ben2018} If $[C]$ is periodic of period $p$ then by \cite[Corollary 1.3.43]{Ben2018} $P^{n}(C)$ is a $\Lambda\text{-}k[T,T^{-1}]$-bimodule
where $T$ acts on the right by $ b _{i}\mapsto b _{i-p}$. By translational symmetry the map $d_{P(C)}^{n}:P^{n}(C)\rightarrow P^{n+1}(C)$ is $\Lambda\otimes_{k}k[T,T^{-1}]$-linear. For a $k[T,T^{-1}]$-module $V$ we define $P(C,V)$ by $P^{n}(C,V)=P^{n}(C)\otimes_{k[T,T^{-1}]}V$
and $d_{P(C,V)}^{n}=d_{P(C)}^{n}\otimes\mathrm{id}_{V}$ for each $n\in\mathbb{Z}$. 

\cite[Definition 1.3.48]{Ben2018} A \textit{string complex} has the form $P(C)$ where $[C]$ is aperiodic. If $V$ is a $k[T,T^{-1}]$-module we call $P(C,V)$ a \textit{band complex} provided $[C]$ is a periodic homotopy $\mathbb{Z}$-word and $V$ is an indecomposable
$k[T,T^{-1}]$-module.
\end{definition}
\begin{example}Let $0\neq\eta\in k$, $[C]={}^{\infty}([x][y^{-1}][g][z^{-1}])^{\infty}$ and let $V=k^{2}$ be the $k[T,T^{-1}]$ module given by $T(v_{1},v_{2})=(\eta v_{1}+v_{2},\eta v_{2})$. Here the band complex $P(C,V)$ is depicted by
\[
\xymatrix@R=0.5em@C=0.5em{&
\Lambda e_{4}\oplus\Lambda e_{4}\ar@/_{1pc}/[ddl]^{z}_(0.7){A=\small{\begin{pmatrix}\eta & 1\\
0 & \eta
\end{pmatrix}}^{-1}}\ar[ddrrrr]|>>>>>>>>>>>>>>>>>>>>>{\hole}_(0.55){g} & & &
\Lambda e_{0}\oplus\Lambda e_{0}\ar@/^{0.6pc}/[ddllll]^(0.5){x}\ar@/^{1pc}/[ddr]^{y} &  & & & P^{-1}(C,V)\ar[dd]^{d_{P(C,V)}^{-1}}\\
&  &  &  &  &  &  & & &\\
\Lambda e_{0}\oplus\Lambda e_{0}&  &  &  &  &  \Lambda e_{1}\oplus\Lambda e_{1} & & & P^{0}(C,V) 
}
\]
See \cite[Definition 2]{BekMer2003} and \cite[Definition 3.7 and Remark 3.8]{Ben2016} for further details.
\end{example}
\section{Compact objects in the homotopy category.}\label{cmpctgentle}Our goal for \S\ref{cmpctgentle} is to describe the compact objects in $\mathcal{K}(\Lambda\text{-}\mathbf{Proj})$ by combining Theorems \ref{neemantheorem} and \ref{theorem.1.1}
Recall, by Theorem \ref{neemantheorem}, that the full subcategory of compact objects in $\mathcal{K}(\Lambda\text{-}\mathbf{Proj})$ is the category $\mathcal{K}^{+,b}(\Lambda\text{-}\mathbf{proj})$ of left-bounded complexes with bounded cohomology and finitely generated homogeneous components. With a view toward describing the compact objects in $\mathcal{K}(\Lambda\text{-}\mathbf{Proj})$, we explain why it suffices to restrict our focus to string and band complexes. 
\begin{theorem}
\label{theorem.1.1}\emph{\cite[Theorem 2.0.1]{Ben2018}} The following statements hold.
\begin{enumerate}
\item Every object in $\mathcal{K}(\Lambda\text{-}\boldsymbol{\mathrm{proj}})$ is isomorphic to a (possibly infinite) coproduct of shifts of string complexes $P(C)$ and shifts of band complexes $P(C,V)$.
\item Each shift of a string or band complex is an indecomposable object in $\mathcal{K}(\Lambda\text{-}\boldsymbol{\mathrm{Proj}})$.
\end{enumerate}
\end{theorem}
To describe the string and band complexes in the category $\mathcal{K}^{+,b}(\Lambda\text{-}\mathbf{proj})$ we begin by computing the kernel of a differential for a string complex.
\begin{definition}
\label{definition.10.5}Let $[C]$ be a generalised $I$-word. For each $i\in I$ define a path $\kappa(i)$ by:

(a) $\kappa(i)=e_{v_{C}(i)}$ if ($i-1\notin I$ or $[C_{i}]=[\gamma^{-1}]$)
and ($i+1\notin I$ or $[C_{i+1}]=[\lambda]$);

(b) $\kappa(i)= \mathrm{f}( \gamma)$ if $i\pm1\in I$ and $[C_{i}][C_{i+1}]=[\gamma^{-1}][\lambda^{-1}]$;

(c) $\kappa(i)= \mathrm{f}( \lambda)$ if $i\pm1\in I$ and $[C_{i}][C_{i+1}]=[\gamma][\lambda]$;

(d) $\kappa(i)=\beta\in\mathbf{A}$ if $i-1\notin I\ni i+1$, $[C_{i+1}]=[\lambda^{-1}]$ and $\beta\mathrm{l}(\lambda)\in\mathcal{J}$;

(e) $\kappa(i)=\alpha\in\mathbf{A}$ if $i+1\notin I\ni i-1$, $[C_{i}]=[\gamma]$ and $\alpha \mathrm{l}(\gamma)\in\mathcal{J}$; and

(f) $\kappa(i)=0$ if $i\pm1\in I$ and $[C_{i}][C_{i+1}]=[\gamma][\lambda^{-1}]$.

Note that for any $i\in I$ exactly one of the ((a), (b), (c), (d),
(e) and (f)) is true. 

We say that the $i^{\mathrm{th}}$\textit{ kernel part} is: \textit{full}
in case (a); a \textit{left} (resp. \textit{right}) \textit{arm} in
case (b) (resp. (c)); a \textit{left} (resp. \textit{right}) \textit{peripheral
arm} in case (d) (resp. (e)); and $0$ in case (f). 
\end{definition}
\begin{corollary}
\label{corollarly.10.3}\emph{\cite[Corollary 2.7.8]{Ben2018}} Let $[C]$ be a generalised $I$-word. For any
\emph{$n\in\mathbb{Z}$} we have \emph{$\mathrm{ker}(d_{P(C)}^{n})=\bigoplus_{i\in\mu_{C}^{-1}(n)}\Lambda\kappa(i)b_{i}$}. 
\end{corollary}
We now recall the way in which generalised words may be composed.
\begin{definition}\label{generalisedsign} Let $h(\alpha^{-1})=t(\alpha)$ and $t(\alpha^{-1})=h(\alpha)$ for each $\alpha\in\mathbf{A}$. Choose \textit{sign} $s(l)\in\{\pm1\}$ for each symbol $l$ of the form $\alpha$ or $\alpha^{-1}$ with $\alpha\in\mathbf{A}$: such that if distinct such symbols $l$ and $l'$ have the same head, they have the same sign if and only if $\{l,l'\}=\{\alpha^{-1},\beta\}$
with $\alpha\beta\in\mathcal{J}$. Now let $s([\gamma])=s( \mathrm{f}( \gamma)^{-1})$ and $s([\gamma^{-1}])=-s( \mathrm{l}( \gamma))$, and for a generalised $I$-word $[C]$ with $\{0\}\neq I\subseteq \mathbb{N}$ we let $h([C])=h([C_{1}])$ and $s([C])=s([C_{1}])$. For any vertex $v$ let $s([1_{v,\pm1}])=\pm1$. 

Let $D$ and $E$ be homotopy words where $I_{[D^{-1}]}\subseteq\mathbb{N}$ and $I_{[E]}\subseteq\mathbb{N}$. If $u=h([D^{-1}])$ and $\epsilon=-s([D^{-1}])$ let $[D][ 1 _{u,\epsilon}]=[D]$. If $v=h([E])$ and $\delta=s([E])$ we let $ [1 _{v,\delta}][E]=[E]$.
The \textit {composition} $[D][E]$ is the concatenation of the generalised letters in $[D]$ with those in $[E]$. The result is a generalised word
if and only if $h([D^{-1}])=h([E])$ and $s([D^{-1}])=-s([E])$ \cite[Proposition 2.1.13]{Ben2018}. If $[D]=\dots [D_{-1}][D_{0}]$ is a $-\mathbb{N}$-word and $[E]=[E_{1}][E_{2}]\dots$ is an $\mathbb{N}$-word, write $[D][E]=\dots [D_{-1}][D_{0}]:[E_{1}][E_{2}]\dots$.
\end{definition} 
 To streamline Proposition \ref{proppss} we use the following notation.
\begin{definition}\label{nesw}Let $[C]$ be a finite generaised word. Let $\mathcal{W}^{+}_{\leftarrow}(C)$ be the (potentially empty) set of generalised $I$-words of the form 
\[
[B^{+}]=
\begin{cases}
[\alpha_{m}]\,\dots \,[\alpha_{1}] & (\mbox{if }I=\{0,\dots,m\}\mbox{ for some }m>0)\\
\dots \, [\alpha_{3}] \, [\alpha_{2}] \, [\alpha_{1}] & (\mbox{if }I=-\mathbb{N})
\end{cases}
\]
where $\alpha_{i}\in\mathbf{A}$ for all $i>0$ in $I$, and the concatenation $[B^{+}][C]$ is again a generalised word. If $\mathcal{W}^{+}_{\leftarrow}(C)=\emptyset$ let $[C(\swarrow)]=[C]$, and otherwise let $[C(\swarrow)]=[B][C]$ where the generalised $I$-word $[B]\in\mathcal{W}^{+}_{\leftarrow}(C)$ is unique such that for any generalised $I'$-word $[B']\in\mathcal{W}^{+}_{\leftarrow}(C)$ we have $I'\subseteq I$. 

Dually let $\mathcal{W}^{-}_{\leftarrow}(C)$ be the set of generalised $I$-words of the form 
\[
[B^{-}]=
\begin{cases}
[\beta^{-1}_{m}]\,\dots\,[\beta_{1}^{-1}] & (\mbox{if }I=\{0,\dots,m\}\mbox{ for some }m>0)\\
\dots \, [\beta^{-1}_{3}]\, [\beta^{-1}_{2}] \, [\beta^{-1}_{1}] & (\mbox{if }I=-\mathbb{N})
\end{cases}
\]
where $\beta_{i}\in\mathbf{A}$ for all $i>0$ in $I$ and $[B^{-}][C]$ is a generalised word. As above we let $[C(\nwarrow)]=[C]$ if $\mathcal{W}^{-}_{\leftarrow}(C)=\emptyset$, and otherwise $[C(\nwarrow)]=[B][C]$ where $[B]\in\mathcal{W}^{-}_{\leftarrow}(C)$ is maximal. 

Similarly we can define the generalised words $[C(\searrow)]$ and $[C(\nearrow)]$ by: letting $\mathcal{W}^{+}_{\rightarrow}(C)$ be the set of generalised $I$-words of the form
 \[
 [D^{+}]=
\begin{cases}
[\gamma_{1}^{-1}]\,\dots\,[\gamma_{m}^{-1}] & (\mbox{if }I=\{0,\dots,m\}\mbox{ for some }m>0)\\
[\gamma_{1}^{-1}]\,[\gamma_{2}^{-1}]\,[\gamma_{3}^{-1}]\,\dots & (\mbox{if }I=\mathbb{N})
\end{cases}
\]
where $\gamma_{i}\in\mathbf{A}$ for all $i>0$ in $I$ and $[C][D^{+}]$ is a generalised word
; and by letting $\mathcal{W}^{-}_{\leftarrow}(C)$ be the set of generalised $I$-words of the form 
\[
 [D^{-}]=
\begin{cases}
[\delta_{1}]\,\dots \,[\delta_{m}] & (\mbox{if }I=\{0,\dots,m\}\mbox{ for some }m>0)\\
[\delta_{1}]\,[\delta_{2}]\,[\delta_{3}]\,\dots & (\mbox{if }I=\mathbb{N})
\end{cases}
\]
where $\gamma_{i}\in\mathbf{A}$ for all $i>0$ in $I$ and $[C][D^{-}]$ is a generalised word. 

Finally, we define the generalised words $[C(\swarrow\searrow)]$, $[C(\swarrow\nearrow)]$, $[C(\nwarrow\nearrow)]$ and $[C(\nwarrow\searrow)]$ as follows. Let $\rightarrowtail\in\{\searrow,\nearrow\}$. If $\mathcal{W}^{\pm}_{\leftarrow}(C)\neq\emptyset$ then $[C(\leftarrowtail)]=[B^{\pm}][C]$ for appropriate $\leftarrowtail\in\{\swarrow,\nwarrow\}$ and sign $\pm$, and we let $[C(\leftarrowtail\rightarrowtail)]=[B^{\pm}][C(\rightarrowtail)]$. Otherwise let $[C(\leftarrowtail\rightarrowtail)]=[C(\rightarrowtail)]$.

In case $[C(\leftarrowtail)]$ is a generalised $-\mathbb{N}$-word and $[C(\rightarrowtail)]$ is a generalised $\mathbb{N}$-word we index the generalised $\mathbb{Z}$-word by $[C(\leftarrowtail\rightarrowtail)]=[B^{\pm}]:[C(\rightarrowtail)]$ in the above notation.
\end{definition}

\begin{example}\label{neswexample}Consider the gentle algebra $\Lambda$ from Example \ref{runningexample} and the generalised word $[C]=[h^{-1}][g^{-1}][(ft)^{-1}][s^{-1}][r^{-1}][yf][x^{-1}]$ from Example \ref{runningexample2}. Here we have $[C(\swarrow)]={}^{\infty}([r][s][t])[C]$ and $[C(\searrow)]=[C] ([x^{-1}])^{\infty}$, and so
\[
[C(\swarrow\searrow)]={}^{\infty}([r][s][t])\, :\, [h^{-1}]\,[g^{-1}]\,[(ft)^{-1}]\,[s^{-1}]\,[r^{-1}]\,[yf]\, ([x^{-1}])^{\infty}.
\]
We may depict $P=P(C(\swarrow\searrow))$ by
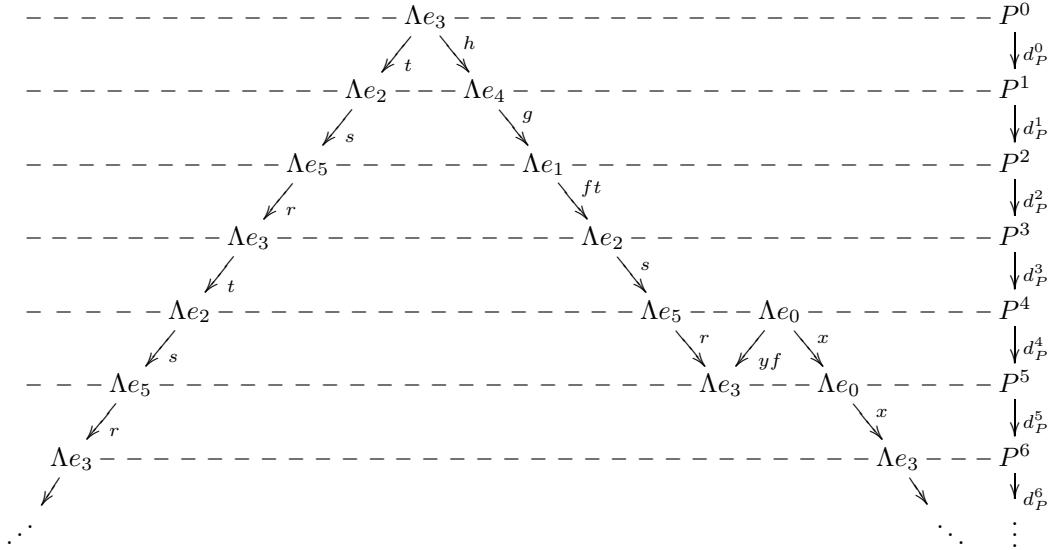
\[
\xymatrix@C=.005em@R=1.2em{ 
& & & & & & &\Lambda e_{3}\ar@{--}[lllllll]\ar[dl]^{t}\ar[dr]^{h} & & & & & & & & & &  & P^{0} \ar[d]^{d_{P }^{0}}\ar@{--}[lllllllllll]\\
& & & & & & \Lambda e_{2}\ar@{--}[llllll]\ar[dl]^{s} & & \Lambda e_{4}\ar@{--}[ll]\ar[dr]^{g} & & & & & & & & & & P^{1} \ar[d]^{d_{P }^{1}}\ar@{--}[llllllllll]\\
& & & & & \Lambda e_{5}\ar@{--}[lllll]\ar[dl]^{r} & & & & \Lambda e_{1}\ar@{--}[llll]\ar[dr]^{ft} & & & & & & & & & P^{2} \ar[d]^{d_{P }^{2}}\ar@{--}[lllllllll]\\
& & & & \Lambda e_{3}\ar@{--}[llll]\ar[dl]^{t} & & & & & & \Lambda e_{2}\ar@{--}[llllll]\ar[dr]^{s} & & & & & & & & P^{3} \ar[d]^{d_{P }^{3}}\ar@{--}[llllllll]\\
& & & \Lambda e_{2}\ar@{--}[lll]\ar[dl]^{s} & & & & & & & & \Lambda e_{5}\ar@{--}[llllllll]\ar[dr]^{r} &  & \Lambda e_{0}\ar@{--}[ll]\ar[dl]^{yf}\ar[dr]^{x} & & & & & P^{4} \ar[d]^{d_{P }^{4}}\ar@{--}[lllll]\\
& & \Lambda e_{5}\ar@{--}[ll]\ar[dl]^{r} & & & & & & & & & & \Lambda e_{3}\ar@{--}[llllllllll] & & \Lambda e_{0}\ar@{--}[ll]\ar[dr]^{x}  & & & & P^{5} \ar[d]^{d_{P }^{5}}\ar@{--}[llll]\\
& \Lambda e_{3}\ar@{--}[l]\ar[dl] & & & & & & & & & & & & & &  \Lambda e_{3}\ar@{--}[llllllllllllll]\ar[dr] & & & P^{6} \ar[d]^{d_{P }^{6}}\ar@{--}[lll]\\
\iddots & & & & & & & & & & & & & & & & \ddots &  &  \vdots
}\]
\end{example}

We now describe the string and band complexes in $\mathcal{K}^{+,b}(\Lambda\text{-}\boldsymbol{\mathrm{proj}})$ by adapting the proof of \cite[Lemma 2.7.5]{Ben2018}. The proof is essentially the same, but for completeness we repeat it.
\begin{proposition}\label{proppss}
The following statements hold.
\begin{enumerate}
\item A string complex lies in $\mathcal{K}^{+,b}(\Lambda\text{-}\boldsymbol{\mathrm{proj}})$ if and only it has the form $P(C)$, $P(C(\swarrow))$, $P(C(\searrow))$ or $P(C(\swarrow\searrow))$ for some finite generalised word $[C]$.
\item A shift of a band complex $P(D,V)$ lies in $\mathcal{K}^{+,b}(\Lambda\text{-}\boldsymbol{\mathrm{proj}})$ if and only if the indecomposable $k[T,T^{-1}]$-module $V$ is finite-dimensional over $k$.
\end{enumerate}
\end{proposition}
\begin{proof}(i) Let $P(A)$ be the string complex in question, where $[A]$ is an aperiodic generalised $I$-word. Without loss of generality we can assume $I=\mathbb{Z}$. 

Suppose firstly that there is a sequence $(i_{n}\mid n\in\mathbb{N})\in I^{\mathbb{N}}$ such that the $i_{n}^{\text{th}}$ kernel part is full for each $n$. Since $P(A)$ is bounded below $\{\mu_{A}(i_{n})\mid n\in\mathbb{N}\}$ does not have
a lower bound. This means there is a subsequence $(i_{n(r)}\mid r\in\mathbb{N})$
of $(i_{n}\mid n\in\mathbb{N})$ such that $\mu_{A}(i_{n(r)})>\mu_{A}(i_{n(r+1)})$
for all $r$. By definition, for each $r$ we have $b_{i_{n(r)}}\notin\mathrm{im}(d_{P(A)})$, and the assumption on $(i_{n})$ gives $b_{i_{n(r)}}\in\mathrm{ker}(d_{P(A)})$, which contradicts that $P(A)$ has bounded cohomology.

Hence we have shown that there are no sequences $(i_{n}\mid n\in\mathbb{N})\in I^{\mathbb{N}}$
such that the $i_{n}^{\mathrm{th}}$ kernel part is full for each
$n$. So we can choose $l\in I$ such that $[A_{l+1}][A_{l+2}]\dots=[\lambda^{-1}_{1}][\lambda^{-1}_{2}]\dots$
for a sequence of paths $\lambda_{j}\in\mathbf{P}$ where $ \mathrm{f}( \lambda_{j}) \mathrm{l}( \lambda_{j+1})=0$
for each $j\geq1$. Now choose $q\in\mathbb{Z}$ such that $\mathrm{im}(d_{P(A)}^{p-1})=\mathrm{ker}(d_{P(A)}^{p})$
for all $p<q$. Choose $t>l$ such that $\mu_{A}(i)<q$
for each $i>t$. 

If there is some $j>t-l$ where $\lambda_{j}$ has
length greater than $1$ then $d_{P(A)}(b_{l+j})=\lambda_{j}b_{l+j+1}$
and so $ \mathrm{f}( \lambda_{j})b_{l+j+1}\notin\mathrm{im}(d_{P(A)})$.
By Corollary \ref{corollarly.10.3} we have $ \mathrm{f}( \lambda_{j})b_{l+j+1}\in\mathrm{ker}(d_{P(A)})$,
which contradicts that $\mathrm{im}(d_{P(A)}^{n-1})=\mathrm{ker}(d_{P(A)}^{n})$ where $n=\mu_{A}(l+j+1)$.

Hence $\lambda_{j}$ is an arrow for each $j>t-l$. Now let $\gamma_{h}=\lambda_{j+h}$
for each integer $h>0$. Since the quiver $Q$ is finite there is some $h>0$
such that $\gamma_{h}=\gamma_{h+n}$ for some $n>0$, which means
$\gamma_{h}=\gamma_{h+n}$ for each $h>0$. Altogether we have $[A_{t+1}][A_{t+2}]\dots=([\gamma^{-1}_{n}]\dots[\gamma_{1}^{-1}])^{\infty}$ for some $n$. Similarly one can show that there are integers $r,m$ for which $\dots[A_{r-1}][A_{r}]={}^{\infty}([\alpha_{m}]\dots[\alpha_{1}])$ where $\alpha_{i}\in\mathbf{A}$ for each $i$. We now have that $[A]=[C(\swarrow\searrow)]$ where $[C]=[A_{r}]\dots[A_{t+1}]$.

(ii) Any band complex $P(C,V)$ is a bounded complex, whose homogeneous component in degree $n$ is a coproduct of (indecomposable projective modules of the form $\Lambda e_{v}$) which runs through the direct product of $\mu_{C}^{-1}(n)\cap[0,p-1]$ and a $k$-basis of $V$. Hence $\mathrm{dim}_{k}(P^{n}(C,V))<\infty$ if and only if the above product is finite, and the set $\mu_{C}^{-1}(n)\cap[0,p-1]$ is always finite.
\end{proof}
Recall, in the notation from Definition \ref{compactob}, that $\mathcal{K}(\Lambda\text{-}\boldsymbol{\mathrm{Proj}})^{c}$ is the full triangulated subcategory of $\mathcal{K}(\Lambda\text{-}\boldsymbol{\mathrm{Proj}})$ consisting of compact objects. 
\begin{corollary}\label{compactcor}The following statements hold.
\begin{enumerate}
\item Every object in $\mathcal{K}(\Lambda\text{-}\boldsymbol{\mathrm{Proj}})^{c}$ is a coproduct of shifts of string and band complexes. 
\item The string complexes which lie in $\mathcal{K}(\Lambda\text{-}\boldsymbol{\mathrm{Proj}})^{c}$ are precisely those of the form $P(C)$, $P(C(\swarrow))$, $P(C(\searrow))$ or $P(C(\swarrow\searrow))$ for some finite generalised word $[C]$.
\item The string complexes which lie in $\mathcal{K}(\Lambda\text{-}\boldsymbol{\mathrm{Proj}})^{c}$ are precisely those of the form $P(D,V)$ where the indecomposable $k[T,T^{-1}]$-module $V$ is finite-dimensional over $k$.
\end{enumerate}
\end{corollary}
\begin{proof}Part (i) follows from Theorems \ref{neemantheorem} and \ref{theorem.1.1}. Parts (ii) and (iii) follow from Theorem \ref{neemantheorem} and Proposition \ref{proppss}.
\end{proof}
\begin{remark}Recall that Theorem \ref{neemantheorem} is a generalisation by Neeman \cite{Nee2008} of a result due to J\o rgenson \cite[Theorem 3.2]{Jor2005}. The benefit of Neemans generality is that one only requires a right coherent ring. We assert that the full generality of Neeman's result was not necessary here. To start, note that since $\Lambda$ is a finite-dimensional $k$-algebra, the jacobson radical $\mathrm{rad}(\Lambda)$ is nilpotent, and the quotient ring $\Lambda/\mathrm{rad}(\Lambda)$ is semisimple. In particular, by a well-known equivalence due to Bass \cite[Theorem P]{Bas1960}, for all integers $n\geq0$, the limit of modules of projective-dimension at most $n$ must have projective-dimension at most $n$. 

By a well-known result of Lazard \cite[Th\'eor\`eme 1.2]{Laz1969}, over any unital ring, every flat module is a limit of free modules, which are projective. Altogether this shows any flat $\Lambda$-module is projective. This shows that the criterion from \cite[Setup 2.1]{Jor2005} are met. Since $\Lambda$ is right noetherian the duality $(-)^{\star}$ given by the functor $\mathrm{Hom}_{\Lambda\text{-}\mathbf{Mod}}(-,\Lambda):\Lambda\text{-}\mathbf{Mod}\to \mathbf{Mod}\text{-}\Lambda$ (taking left modules to right modules) restricts to a functor $(-)^{\star}|:\Lambda\text{-}\mathbf{mod}\to \mathbf{mod}\text{-}\Lambda$ between full subcategories of finitely generated modules. By \cite[Theorem 3.2]{Jor2005} we then have that $(-)^{\star}|$ defines a triangle equivalence $\mathcal{K}^{c}(\Lambda\text{-}\mathbf{Proj})\rightarrow \mathcal{K}^{-,b}(\mathbf{proj}\text{-}\Lambda)$, and from here one may give another proof of \ref{compactcor} as was done, for example, in the proof of \cite[Proposition 3.6]{ArnLakPauPre2017}.
\end{remark}
\section{Linear relations and homotopy words.}\label{linrels}
The proof of \cite[Theorem 1.1]{BenCra2018} uses the so-called functorial filtrations method (as discussed in \S\ref{intro}), whch iwas written in the language of \textit{additive relations} in the sense of Mac Lane \cite{Mac1961}. The aforementioned method depends on a certain splitting result for finite-dimensional $k$-linear relations, see \cite[Theorem 3.1]{GelPon1968}, \cite[\S 2]{Rin1975} and \cite[\S7]{Gab1972}.
\begin{definition}\label{defrelations}
Given $k$-vector spaces $V$ and $W$ a \emph{linear relation from} $V$ \emph{to} $W$ (or \emph{on} $V$ if $W=V$) is a $k$-subspace $\mathscr{R}$ of the coproduct $V\oplus W$. This generalises the graph of a $k$-linear map $V\rightarrow W$. The category $k\text{-}\bf{Rel}$ of linear relations has as objects the pairs $(V,\mathscr{R})$ where $\mathscr{R}$ is a relation on $V$, and has morphisms
$(V,\mathscr{R})\to (W,\mathscr{S})$ given by $k$-linear maps $f:V\to W$
with $(f(x),f(y))\in \mathscr{S}$ for all $(x,y)\in \mathscr{R}$. 
%
Let $\Gamma$ be the Kronecker quiver, given by two arrows $p$ and $q$ with common tail $u$ and common head $v$, and let $k\Gamma$ be the path algebra.

As is well-known, there is an equivalence between the category $k\Gamma\text{-}\bf{Mod}$ of left $k\Gamma$-modules to the category $k\text{-}\textbf{Rep}(\Gamma)$ of $k$-representations $(\phi_{p},\phi_{q}:L_{u}\rightarrow L_{v})$ of $\Gamma$.
Any relation $\mathscr{R}$ on $V$ defines an object $(\pi_{p},\pi_{q}:\mathscr{R}\rightarrow V)$ of $k\text{-}\textbf{Rep}(\Gamma)$ where $\pi_{p}$ and $\pi_{q}$ are the compositions of the inclusion $\mathscr{R}\subseteq V\oplus V$ with the first and second projections $V\oplus V\rightarrow V$. In this way there is a fully-faithful additive functor $k\text{-}\textbf{Rel}\to k\Gamma\text{-}\textbf{Mod}$
whose essential image we denote $k\text{-}\textbf{Rep}(\Gamma)_{\mathrm{rel}}$.
\end{definition}
This shows that one may equip $k\text{-}\bf{Rel}$ with various structural properties inherited from the category $k\Gamma\text{-}\textbf{Mod}$. We document some of the said properties below.
\begin{remark}\label{co/prodclsd}Note $k\Gamma\text{-}\textbf{Mod}_{\mathrm{rel}}$ consists of modules $X$ where $e_{u}X\to e_{v}X\oplus e_{v}X$, $x\mapsto (px,qx)$ is injective. This property is closed under taking equalisers,
products 
and coproducts.
Conseuently, the category $k\text{-}\bf{Rel}$ has all limits and all coproducts, computed by passing back and forth between categories of relations modules. For example, a sequence of relations
\[
0 \to (U,\mathscr{R}) \to (V,\mathscr{S}) \to (W,\mathscr{T})\to 0
\]
is exact provided
that the underlying sequences of $k$-vector spaces \[
0\to U \to V \to W \to 0\text{
 and }0\to \mathscr{R} \to \mathscr{S} \to \mathscr{T} \to 0
\]
are exact. For a set $I$ and an object $(V_{i},\mathscr{R}_{i})$ of $k\text{-}\bf{Rel}$ for each $i$ the set of pairs $((v_{i}),(v'_{i}))$ with $(v_{i},v_{i}')\in \mathscr{R} _{i}$ for each $i$ defines a the product $\prod (V_{i},\mathscr{R}_{i})$ of the objects $(V_{i},\mathscr{R}_{i})$.
Similarly the coproduct $\bigoplus (V_{i},\mathscr{R}_{i})$ is given by the relation on $\bigoplus V_{i}$ consisting of pairs $((v_{i}),(v'_{i}))$ as above, but where additionally $v_{i}=v_{i}'=0$ for all but finitely many $i$. 
\end{remark}
Recall that: an embedding of modules (over a fixed unital ring) is \textit{pure} provided it remains an embedding under any tensor product functor; a module is \textit{pure}-\textit{injective} if it is injective with respect to pure monomorphisms, and $\Sigma$-\textit{pure}-\textit{injective} modules are those for which any small coproduct of copies of it is pure-injective. The coproduct over a singleton shows $\Sigma$-pure-injectives are pure-injective. By the equivalences of (ii) and (vi) in \cite[Theorem 7.1]{JenLen1989}, and of (i) and (ii) in \cite[Theorem 8.1]{JenLen1989}, we can use a categorical definition of pure-injectivity in categories of linear relations.
\begin{definition}
By Remark \ref{co/prodclsd} the category $k\text{-}\bf{Rel}$ has small products and small coproducts. Fix an object $(V,\mathscr{R})$ of $k\text{-}\bf{Rel}$. The universal properties of the products and coproduct define a \emph{summation} map $\sigma_{I}:\bigoplus_{i}(V,\mathscr{R})\to(V,\mathscr{R})$ and a \textit{canonical} map $\iota_{I}:\bigoplus_{i}(V,\mathscr{R})\to\prod_{i}(V,\mathscr{R})$.
We say $(V,\mathscr{R})$ is \textit{pure}-\textit{injective} if, for any set $I$, $\sigma_{I}$ factors through $\iota_{I}$ to extend to a map $\prod_{i} (V,\mathscr{R})\to(V,\mathscr{R})$. We say $(V,\mathscr{R})$ is $\Sigma$-\emph{pure}-\emph{injective} if, for any set $I$, $\sigma_{I}$ is a section.
\end{definition}
Corollary \ref{meandbill} gives a relationship between $\Sigma$-pure-injective linear relations and $\Sigma$-pure-injective $k[T,T^{-1}]$-modules. To state Corollary \ref{meandbill}  we need another definition.
\begin{definition}\label{defsplit}For an object $(V,\mathscr{R})$ of $k\text{-}\bf{Rel}$ let $\mathscr{R}v = \{ w\in W \colon (v,w)\in \mathscr{R} \}$ for any $v\in V$, and for a subset $U\subseteq V$ let $\mathscr{R}U$ be the union $ \bigcup \mathscr{R}u$ over $u\in U$. When $\mathscr{R}$ is the graph of a map $f$ then $\mathscr{R}U$ is the image of $U$ under $f$. Furthermore, let
\[
\begin{array}{c}
\mathscr{R}'' = \{ v\in V : \exists \,(v_{n})\in V^{\mathbb{N}}\text{ with }(v_{n},v_{n+1})\in \mathscr{R}\text{ and }v=v_{0}\},
\\
\mathscr{R}' = \{ v\in V : \exists \,(v_{n})\in V^{\mathbb{N}}\text{ with }(v_{n},v_{n+1})\in \mathscr{R},\,v=v_{0}\text{ and }v_{n} = 0\text{ for }n\gg0\}.
\end{array}
\]
Define subspaces $\mathscr{R}^{\flat}\subseteq \mathscr{R}^{\sharp}\subseteq V$ by
\[
\begin{array}{c}
 \mathscr{R}^\sharp = \mathscr{R}'' \cap (\mathscr{R}^{-1})'',\text{ and }\mathscr{R}^\flat = \mathscr{R}'' \cap (\mathscr{R}^{-1})'+(\mathscr{R}^{-1})'' \cap \mathscr{R}'.
\end{array}
\]
By \cite[Lemma 4.5]{Cra2018} the quotient $\mathscr{R}^{\sharp}/\mathscr{R}^{\flat}$ is a $k[T,T^{-1}]$-module with the action of $T$ given by 
\[
T(v+\mathscr{R}^{\flat})=w+\mathscr{R}^{\flat}\text{ if and only if }w\in \mathscr{R}^{\sharp}\cap(\mathscr{R}^{\flat}+\mathscr{R}v).
\]
We say an object $(V,\mathscr{S})$ of $k\text{-}\bf{Rel}$ is \emph{automorphic} if both projection maps $\mathscr{S}\to V$ are isomorphisms, and that $(V,\mathscr{R})$ is \emph{split} provided that there is a subspace $W$ of $V$ such that $\mathscr{R}^\sharp = \mathscr{R}^\flat \oplus W$ and such that the object $(W,\mathscr{R}\cap(W\oplus W))$ of $k\text{-}\bf{Rel}$ is automorphic \cite[\S 4]{Cra2018}. 
\end{definition}
In the sense of Ringel \cite[\S 2]{Rin1975}, $\mathscr{R}'$ is equal to the \emph{stable kernel} $\bigcup _{n>0}\mathscr{R}^{n}0$, and $\mathscr{R}''$ is a subspace of the \emph{stable image} $\bigcap _{n>0}\mathscr{R}^{n}V$. Furthermore if $\mathrm{dim}_{k}(V)<\infty$ then the inclusion of $\mathscr{R}''$ in the stable image is an equality \cite[Lemma 4.2]{Cra2018}. In joint work with Crawley-Boevey \cite{BenCra2018} we considered $k$-linear relations $(V,\mathscr{R})$ as Kronecker modules, via the first and second projections of $\mathscr{R}$ onto $V$, in order to prove the following.
\begin{corollary}\label{meandbill}\emph{\cite[Corollary 1.3]{BenCra2018}} Let $(V,\mathscr{R})$ be $\Sigma$-pure-injective object of $k\text{-}\bf{Rel}$. Then $(V,\mathscr{R})$ is split and $\mathscr{R}^{\sharp}/\mathscr{R}^{\flat}$ is
a $\Sigma$-pure-injective $k[T,T^{-1}]$-module.
\end{corollary}
We now recall the language of homotopy words developed in \cite{Ben2018}. The reason for introducing this language is to apply the functorial filtrations method. Each homotopy word will give rise to a linear relation on the underlying subspace of the complex.
\begin{definition}\label{homotopywords}\cite[Definition 1.3.26]{Ben2018} A \textit{homotopy letter $q$} is one of $\gamma$, $\gamma^{-1}$, $d_{\alpha}$,
or $d_{\alpha}^{-1}$ for $\gamma\in\mathbf{P}$ and an arrow $\alpha$. Those of the form $\gamma$ or $d_{\alpha}$ will be called \textit{direct}, and those of the form $\gamma^{-1}$ or $d_{\alpha}^{-1}$ will be called \textit{inverse}. The \textit{inverse $q^{-1}$} of a homotopy letter \textit{$q$} is defined by setting $(\gamma)^{-1}=\gamma^{-1}$, $(\gamma^{-1})^{-1}=\gamma$, $(d_{\alpha})^{-1}=d_{\alpha}^{-1}$ and $(d_{\alpha}^{-1})^{-1}=d_{\alpha}$. 

Again let $I$ be one of $\{0,\dots,m\}$, $\mathbb{N}$, $-\mathbb{N}$ or $\mathbb{Z}$. For $I\neq\{0\}$ by a \textit{homotopy} $I$-\textit{word} we mean a sequence of homotopy letters 
\[
C=\begin{cases}
l_{1}^{-1}r_{1}\dots l_{m}^{-1}r_{m} & (\mbox{if }I=\{0,\dots,m\})\\
l_{1}^{-1}r_{1}l_{2}^{-1}r_{2}\dots & (\mbox{if }I=\mathbb{N})\\
\dots l_{-1}^{-1}r_{-1}l_{0}^{-1}r_{0} & (\mbox{if }I=-\mathbb{N})\\
\dots l_{-1}^{-1}r_{-1}l_{0}^{-1}r_{0}\mid l_{1}^{-1}r_{1}l_{2}^{-1}r_{2}\dots & (\mbox{if }I=\mathbb{Z})
\end{cases}
\]
(which will be written as $C=\dots l_{i}^{-1}r_{i}\dots$ to save
space) such that: 
\begin{enumerate}
\item any homotopy letter in $C$ of the form $l_{i}^{-1}$ (respectively $r_i$) is inverse (respectively direct);
\item any sequence of 2 consecutive letters in $C$, which is of the form $l_{i}^{-1}r_{i}$, is one of
$\gamma^{-1}d_{\mathrm{l}(\gamma)}$ or $d_{\mathrm{l}(\gamma)}^{-1}\gamma$ for some $\gamma\in\mathbf{P}$; and 
\item given we let $[C_{i}]=[\gamma^{-1}]$ if $l_{i}^{-1}r_{i}=d_{\mathrm{l}(\gamma)}^{-1}\gamma$
and $[C_{i}]=[\gamma]$ if $l_{i}^{-1}r_{i}=\gamma^{-1}d_{\mathrm{l}(\gamma)}$, then $[C]=\dots [C_{i}]\dots$ defines a generalised $I$-word.
\end{enumerate}
For $I=\{0\}$ there are \textit{trivial homotopy words}
$ 1 _{v,1}$ and $ 1 _{v,-1}$ for each vertex
$v$.
\end{definition}
\begin{notation}
Let $\mathcal{C}_{\mathrm{min}}(\Lambda\text{-}\mathbf{Proj})$
and $\mathcal{K}_{\mathrm{min}}(\Lambda\text{-}\mathbf{Proj})$ be the full subcategories
of $\mathcal{C}(\Lambda\text{-}\mathbf{Proj})$ and $\mathcal{K}((\Lambda\text{-}\mathbf{Proj})$ consisting of \textit{homotopically minimal} complexes: that is, whose objects $M$ are complexes in $\mathcal{C}(\Lambda\text{-}\mathbf{Proj})$ such that $\im(d_{M}^{n})\subseteq\rad(M^{n+1})$
for all $n\in\mathbb{Z}$. 
\end{notation}
\begin{corollary}
\label{isosreflect} The subcategory $\mathcal{K}_{\mathrm{min}}(\Lambda\text{-}\mathbf{Proj})$ of $\mathcal{K}(\Lambda\text{-}\mathbf{Proj})$ is dense, and the quotient functor $\mathcal{C}_{\mathrm{min}}(\Lambda\text{-}\mathbf{Proj})\to \mathcal{K}_{\mathrm{min}}(\Lambda\text{-}\mathbf{Proj})$ reflects isomorphisms.
\end{corollary}
\begin{proof}Since $\Lambda$ is a finite-dimensional $k$-algebra, the jacobson radical $\mathrm{rad}(\Lambda)$ is nilpotent, and the quotient ring $\Lambda/\mathrm{rad}(\Lambda)$ is semisimple. This means $\Lambda$ is a perfect ring, and consequently every object in the category $\Lambda\text{-}\mathbf{Mod}$ of $\Lambda$-modules has a projective cover. Hence the first assertion follows, for example, by \cite[Corollary 3.2.25]{Ben2018}.

Now let $t:M\to N$ be an isomorphism in the category $\mathcal{K}_{\mathrm{min}}(\Lambda\text{-}\mathbf{Proj})$ with inverse $s:N\to M$. Write $\tau:M\to N$ and $\sigma:N\to M$ for the corresponding morphisms in the category $\mathcal{C}_{\mathrm{min}}(\Lambda\text{-}\mathbf{Proj})$. Consider the induced morphisms $\bar{\tau}^{n}:M^{n}/\mathrm{rad}(M^{n})\to N^{n}/\mathrm{rad}(N^{n})$ and $\bar{\sigma}^{n}:N^{n}/\mathrm{rad}(N^{n})\to M^{n}/\mathrm{rad}(M^{n})$ of $\Lambda$-modules (for each $n\in\mathbb{Z}$). By construction the morphisms $\sigma\tau-\mathrm{id}_{M}$ and $\tau\sigma-\mathrm{id}_{N}$ in $\mathcal{C}_{\mathrm{min}}(\Lambda\text{-}\mathbf{Proj})$ are null-homotopic.  

Since $M$ and $N$ are homotopically minimal this means $\bar{\tau}^{n}$ is an isomorphism with inverse $(\bar{\tau}^{n})^{-1}=\bar{\sigma}^{n}$. Since $\Lambda$ is a perfect ring it must be a smeiperfect ring. By \cite[Remark 3.11]{Ben2016} this means that each of the morphisms $\tau^{n}$ is an isomorphism, and so $\tau$ is an isomorphism. This gives the second assertion.
\end{proof}
\begin{definition}\label{dmaps}
\cite[Lemma 2.1.1]{Ben2018} Let $M$ be an object of $\mathcal{K}_{\mathrm{min}}(\Lambda\text{-}\mathbf{Proj})$ and $v$ be a vertex. Write $d_{M}$ for the $\Lambda$-module endomorphism of $M=\bigoplus_{i\in\mathbb{Z}}M^{i}$ given by $m\mapsto\sum_{i}d^{i}_{M}(m)$. Let $d_{v,M}\in\mathrm{End}_{k}(e_{v}M)$ be the restriction of $d_{M}$. By \cite[Lemma 5]{BekMer2003} we have $e_{v}\mathrm{rad}(M)=\bigoplus \beta M$ where $\beta M=\{\beta m\mid m\in e_{t(\beta)}M\}$ and $\beta$ runs through $\mathbf{A}(\rightarrow v)$. For any arrow $\alpha\in\mathbf{A}(\rightarrow v)$ let $\pi_{\alpha}:\bigoplus \beta M\to \alpha M$ and $\iota_{\alpha}:\alpha M\to\bigoplus \beta M$ be the canonical retraction and section.

Define the $k$-linear
endomorphism $d_{\alpha,M}$ of $e_{h(\alpha)}M$ by
\[\begin{array}{c}
d_{\alpha,M}(m)=\iota_{\alpha}(\pi_{\alpha}(d_{v,M}(m)))=\alpha m_{\alpha}\text{ where }d_{v,M}(m)=\sum_{\beta\in\mathbf{A}(\rightarrow v)}\beta m_{\beta}.
\end{array}
\]
Consequently, for any vertex $v$, $d_{v,M}=\sum d_{\beta,M}$ where $\beta$ runs through $\mathbf{A}(\rightarrow v)$.
\end{definition}
\begin{remark}It is worth noting some properties of the maps $d_{\alpha,M}$ from Definition \ref{dmaps}. These properties form part of \cite[Lemma 2.1.2]{Ben2018}. For any $\tau\in\mathbf{P}$ and any $x\in e_{t(\tau)}M$:
\begin{enumerate}
\item if there exists $\sigma\in\mathbf{A}$ such that $\tau\sigma\in\mathbf{P}$ then $d_{ \mathrm{l}( \tau),M}(\tau x)=\tau d_{\sigma,M}(x)$;
\item  if $\tau\sigma\notin\mathbf{P}$ for all $\sigma\in\mathbf{A}$ then $d_{ \mathrm{l}( \tau),M}(\tau x)=0$; 
\item if $h(\theta)=h(\tau)$ for
some arrow $\theta\neq \mathrm{l}(\tau)$ then $d_{\theta,M}(\tau x)=0$;
\item if $h(\phi)=h(\tau)$ for some arrow $\phi$ then $d_{\phi,M}d_{ \mathrm{l}( \tau),M}=0$; and
\item if $\tau x\in\mathrm{im}(d_{ \mathrm{l}( \tau),M})$ then $d_{\varsigma,M}(x)=0$ for any $\varsigma\in\mathbf{A}$ where $\tau\varsigma\in\mathbf{P}$.
\end{enumerate}
\end{remark}

\begin{definition}\label{defsubspaces}
Let $M$ be an object of $\mathcal{K}_{\mathrm{min}}(\Lambda\text{-}\mathbf{Proj})$. Let $\gamma\in\mathbf{P}$, $\alpha\in\mathbf{A}$ and let $U$, $V$ and $W$ be $k$-subspaces of $e_{t(\gamma)}M$, $e_{h(\gamma)}M$ and $e_{h(\alpha)}M$ respectively. We define and denote various subspaces of $M$ by
\[
\begin{array}{cc}
\gamma U=\{\gamma m\in e_{h(\gamma)}M\mid m\in U\}, & \gamma^{-1}V=\{m\in e_{t(\gamma)}M\mid\gamma m\in V\},\\
d_{\alpha}W=\{d_{\alpha,M}(m)\in e_{h(\alpha)}M\mid m\in W\}, & d_{\alpha}^{-1}W=\{m\in e_{h(\alpha)}M\mid d_{\alpha,M}(m)\in W\}.
\end{array}
\]
For a vertex $v$ and a subset $X$ of $e_{v}M$ let $ 1 _{v,\pm1}X=X$. When $U=e_{t(\gamma)}M$ let $\gamma M=\gamma U$ and when $U=e_{t(\gamma)}\mathrm{rad}(M)$ let $\gamma\mathrm{rad}(M)=\gamma U$. The subsets $\gamma^{-1}M$, $d_{\alpha}M$, $d^{-1}_{\alpha}M$, $\gamma^{-1}\mathrm{rad}(M)$, $d_{\alpha}\mathrm{rad}(M)$ and $d^{-1}_{\alpha}\mathrm{rad}(M)$ are defined similarly.
\end{definition}
\begin{example}\label{runningexample3}Consider the gentle algebra $\Lambda$ from Example \ref{runningexample} and the generalised word $[C]=[h^{-1}][g^{-1}][(ft)^{-1}][s^{-1}][r^{-1}][yf] [x^{-1}]$ from Example \ref{runningexample2}. The homotopy word of $[C]$ is
\[
C=d_{h}^{-1}h\,\,d_{g}^{-1}g\,\,d_{f}^{-1}(ft)\,\,d_{s}^{-1}s\,\, d_{r}^{-1}r\,\,(yf)^{-1}d_{y}\,\,d_{x}^{-1}x.
\]
Fix an object $M$ of $\mathcal{K}_{\mathrm{min}}(\Lambda\text{-}\mathbf{Proj})$ and a subspace $V$ of $e_{0}M$. By Definition \ref{defsubspaces} we have that $CV$ is given by the set of $m\in e_{3}M$ such that
\[\text{there exists } 
\left(\begin{array}{c}
m_{0}\in e_{3}M,m_{1}\in e_{4}M,\\
m_{2}\in e_{1}M,m_{3}\in e_{2}M,\\
m_{4}\in e_{5}M,m_{5}\in e_{3}M,\\
m_{6}\in e_{0}M,m_{7}\in V.
\end{array}\right)
\text{with}
\left(\begin{array}{c}m=m_{0},d_{h,M}(m_{0})=hm_{1},\\
d_{g,M}(m_{1})=gm_{2},d_{f,M}(m_{2})=ftm_{3},\\
d_{s,M}(m_{3})=sm_{4},d_{r,M}(m_{4})=rm_{5},\\
yfm_{5}=d_{y,M}(m_{6}),d_{x,M}(m_{6})=xm_{7}.
\end{array}\right)
\]
The elements of $CV$ fit into a schema given by the homotopy word $C$. See, for example, \cite[Examples 2.1.6 and 2.1.8]{Ben2018} and \cite[Example 6.5]{Ben2016}.
\end{example}
We now recall how the constructions from Definition \ref{defsubspaces} are compatible, in some sence, with graded subspaces of a complex.
\begin{corollary}
\label{corgradedspace}\emph{\cite[Corollary 2.2.3]{Ben2018}} Let $M$ be an object of $\mathcal{K}_{\mathrm{min}}(\Lambda\text{-}\mathbf{Proj})$, let $n\in\mathbb{Z}$, let $C$ be a homotopy $\{0,\dots,t\}$-word with $h(C)=u$ and $h(C^{-1})=v$, and for all $i\in\mathbb{Z}$ let $X^{i}$ and $Y^{i}$
be subspaces of $e_{v}M^{i}$ and $e_{u}M^{i}$ respectively. Then $Y^{n}\cap C(\bigoplus_{i}X^{i})=Y^{n}\cap CX^{n+\mu_{C}(t)}$. 
\end{corollary}
\section{Compact morphisms of homotopy words.}\label{seccompact}
We can now interpret the notion of the \textit{sign} of a homotopy word, so that it is consistent with Definition \ref{generalisedsign}, as follows.
\begin{definition}\cite[Definition 2.1.11]{Ben2018} If $\gamma\in\mathbf{P}$ let $s(\gamma)=s( \mathrm{l}( \gamma))$ and $s(\gamma^{-1})=s( \mathrm{f}( \gamma)^{-1})$. If $\alpha\in\mathbf{A}$ let $s(d_{\alpha}^{\pm1})=-s(\alpha)$. For a homotopy $I$-word $C$ with $\{0\}\neq I\subseteq \mathbb{N}$, the head and sign of $[C]$ are the head and sign of the first
homotopy letter of $C$. For any vertex $v$ let $s(1_{v,\pm1})=\pm1$
and $h(1_{v,\pm1})=v$. Now let $D$ and $E$ be homotopy words where $I_{D^{-1}}\subseteq\mathbb{N}$ and $I_{E}\subseteq\mathbb{N}$. The \textit {composition} $DE$ is the concatenation of the homotopy letters in $D$ with those in $E$, and the result is a homotopy word
if and only if $[D][E]$ is a generalised word (that is, if and only if $h(D^{-1})=h(E)$ and $s(D^{-1})=-s(E)$). In case $I_{D}=I_{E}=\mathbb{N}$ the bar $\mid$ is placed between $D$ and $E$ just as the colon $:$ is placed between $[D]$ and $[E]$.

We write $\mathcal{W}_{v,\delta}$ for the set of homotopy $I$-words with $I\subseteq \mathbb{N}$, head $v$ and sign $\delta$.
\end{definition}
\begin{assumption}In the remainder of \S\ref{seccompact} let $v$ be a vertex, let $\delta\in\{1,-1\}$ and let $C\in\mathcal{W}_{v,\delta}$ be a homotopy $\{0,\dots,t\}$-word.
\end{assumption}
In what follows, we show that for a fixed object $M$ of $\mathcal{K}_{\mathrm{min}}(\Lambda\text{-}\mathbf{Proj})$ the subspaces $CV$ from Definition \ref{defsubspaces} are given by evaluating functions arising in a pp-definable subgroup of $M$ of a particular sort; see Lemma \ref{noice}. In Definition \ref{compactmorph} we associate, to each such $C$, a morphism in $\mathcal{K}(\Lambda\text{-}\mathbf{Proj})^{c}$. For reasons that become clear in the proof of Corollary \ref{coolbeans}, we require that the sort of the aformentioned pp-definable subgroups depend only on $v$ and $\delta$ (that is, that the domain of the above morphism) is independent of the choice of homotopy word $C\in\mathcal{W}_{v,\delta}$. To proceed we require more notation. 
\begin{notation}\cite[Definition 2.1.17]{Ben2018} If $B=\dots l_{i}^{-1}r_{i}\dots$
is a homotopy word and $i\in I_{B}$, let $B_{i}=l_{i}^{-1}r_{i}$
and $B_{\leq i}=\dots l_{i}^{-1}r_{i}$ given $i-1\in I_{B}$, and
otherwise $B_{i}=B_{\leq i}= 1 _{h(B),s(B)}$. Similarly let $B_{>i}=l_{i+1}^{-1}r_{i+1}\dots$ given $i+1\in I_{B}$ and
otherwise $B_{>i}= 1 _{h(B^{-1}),-s(B^{-1})}$. Similarly we can define the homotopy words $B_{<i}$ and $B_{\geq i}$ with $B_{\leq i}=B_{<i}B_{i}$
and $B_{i}B_{>i}=B_{\geq i}$.  
\end{notation}
Recall Definition \ref{nesw}.
\begin{definition}\label{compactmorph}Let $v$ be a vertex, let $\delta=\pm 1$, let $m\in\mathbb{N}$ and let $C\in\mathcal{W}_{v,\delta}$ be a homotopy $\{0,\dots,t\}$-word with $C^{-1}\in\mathcal{W}_{u,\varepsilon}$. Write $I$ for the subset of $ \mathbb{Z}$ such that $C(\swarrow\searrow)$ is a homotopy $I$-word. If $1_{v,\delta}(\swarrow)$ is a homotopy $-\mathbb{N}$-word let $n=0$, and otherwise choose $n\in I$ (unique) such that $C(\swarrow\searrow)_{\leq n}=1_{v,\delta}(\swarrow)$. By the \textit{compact morphism given by} $C$, we mean the morphism
\[
a\langle C\rangle:P(1_{v,\delta}(\swarrow\searrow))\to P(C(\swarrow\searrow))\text{ in }\mathcal{K}_{\mathrm{min}}(\Lambda\text{-}\mathbf{Proj})
\]
defined, by case analysis, as follows. Let $E=1_{v,\delta}(\swarrow\searrow)$.
\begin{enumerate}
\item If $E=C(\swarrow\searrow)$ let $a\langle C\rangle$ be the identity on $P(E)$.
\end{enumerate}
Suppose now $E\neq C(\swarrow\searrow)$. This means $C(\searrow)=Bl^{-1}rD$ where  either $B=1_{v,\delta}$ or $B=d_{\alpha(1)}^{-1}\alpha(1)\dots d_{\alpha(q)}^{-1}\alpha(q)$ with $\alpha(1),\dots,\alpha(q)\in\mathbf{A}$, and where $l^{-1}r$ does not have the form $d_{\alpha}^{-1}\alpha$ with $\alpha\in\mathbf{A}$. In case  $B=1_{v,\delta}$ let $q=0$.
\begin{enumerate}\setcounter{enumi}{1}
\item If $l^{-1}r=\gamma^{-1}d_{\mathrm{l}(\gamma)}$ let $a\langle C\rangle(b_{i,E})=b_{i,C}$ when $i\leq n+q$, and $a\langle C\rangle(b_{i,E})=0$ otherwise.
\item If $l^{-1}r=d^{-1}_{\mathrm{l}(\gamma)}\gamma$ where $\gamma=\mathrm{l}(\gamma)\lambda$ for some $\lambda\in\mathbf{P}$ let $a\langle C\rangle(b_{i,E})=b_{i,C}$ when $i\leq n+q$, let $a\langle C\rangle(b_{n+q+1,E})=\lambda b_{n+q+1,C}$ and let $a\langle C\rangle(b_{i,E})=0$ otherwise.
\end{enumerate}
\end{definition}
\begin{remark}It is straightword to check that the assignment from Definition \ref{compactmorph} gives a morphism of chain complexes, and hence, as implied, we can and will consider $a\langle C\rangle$ as a morphism in $\mathcal{K}_{\mathrm{min}}(\Lambda\text{-}\mathbf{Proj})$. Furthermore, as our terminology suggests,  $a\langle C\rangle$ defines a morphism in $\mathcal{K}(\Lambda\text{-}\mathbf{Proj})^{c}$, since the string complexes $P(1_{v,\delta}(\swarrow\searrow))$ and $P(C(\swarrow\searrow))$ are compact by Corollary \ref{compactcor}.
\end{remark}
\begin{lemma}\label{nicemaps}Let $B$ be a homotopy $\mathbb{N}$-word and let $a_{l}=a\langle B_{\leq l}\rangle$ for all $l\in\mathbb{N}$. Then $a_{l+1}$ factors through $a_{l}$ for all $l$, and hence for any object $M$ of $\mathcal{K}_{\mathrm{min}}(\Lambda\text{-}\mathbf{Proj})$ we have
\[
Ma_{1}\supseteq Ma_{2} \supseteq \dots\supseteq Ma_{i} \supseteq \dots
\]
\end{lemma}
\begin{proof}Let $B=l_{1}^{-1}r_{1}l^{-1}_{2}r_{2}\dots$. Fix $l\in\mathbb{N}$. Let $B_{\leq l}(\swarrow\searrow)=B'$ and $B_{\leq l+1}(\swarrow\searrow)=B''$. It suffices to assume $1_{v,\delta}(\swarrow\searrow)\neq B''$, since otherwise $a_{ l+1}$  and $a_{l}$ are both just the identity on $P(1_{v,\delta}(\swarrow\searrow))$. Likewise we can assume $B_{\leq l+1}(\searrow)\neq B_{\leq l}(\searrow)$, since otherwise we have $a_{l}=a_{l+1}$. We assume for simplicity that $B_{\leq l+1}(\searrow)=Ed^{-1}_{\mathrm{l}(\gamma)}\gamma D$ for homotopy words $D$ and $E$ where  $E=d_{\alpha(1)}^{-1}\alpha(1)\dots d_{\alpha(q)}^{-1}\alpha(q)$ with $\alpha(1),\dots,\alpha(q)\in\mathbf{A}$, and $\gamma=\mathrm{l}(\gamma)\lambda$ for some $\lambda\in\mathbf{P}$. The other cases are similar.

 Note $B'_{\leq n}=B''_{\leq n}=1_{v,\delta}(\swarrow)$. Without loss of generality assume $n=0$. It is required that we define a morphism $f:P(B')\to P(B'')$. Again, for simplicity we can assume $D=Cd^{-1}_{\mathrm{l}(\eta)}\eta A$ for homotopy words $C$ and $A$ where  $C=d_{\beta(1)}^{-1}\beta(1)\dots d_{\beta(p)}^{-1}\beta(p)$ with $\beta(1),\dots,\beta(q)\in\mathbf{A}$, and $\gamma=\mathrm{l}(\gamma)\lambda$ for some $\mu\in\mathbf{P}$. From here it suffices to let $f(b_{i,B'})=b_{i,B''}$ when $i\leq p+q+1$, let $a\langle C\rangle(b_{p+q+2,B'})=\mu b_{p+q+2,B''}$ and let $a\langle C\rangle(b_{i,E})=0$ otherwise.
\end{proof}
Lemma \ref{noice} is a key realisation toward the proof of Theorem \ref{maincor}. Recall, from Definition \ref{ppdefsub}, that a pp-definable subgroup of $M$ of sort $X$ is some set $Ma$ of morphisms $X\to M$ of the form $fa$, where $a:X\to Y$ is a fixed morphism in $\mathcal{K}(\Lambda\text{-}\mathbf{Proj})^{c}$, and $f:Y\to M$ varies.
\begin{lemma}\label{noice}Let $M$ be an object of $\mathcal{K}_{\mathrm{min}}(\Lambda\text{-}\mathbf{Proj})$. As in Definition \ref{compactmorph}, let $C\in \mathcal{W}_{u,\delta}$ be a homotopy $\{0,\dots,t\}$-word, let $I=I_{C(\swarrow\searrow)}$, let $n=0$ if $1_{v,\delta}(\swarrow)$ is a homotopy $-\mathbb{N}$-word and otherwise choose $n\in I$ with $C(\swarrow\searrow)_{\leq n}=1_{v,\delta}(\swarrow)$. Let $r\in\mathbb{Z}$ and $a=a\langle C\rangle[r]$, the morphism $P(1_{v,\delta}(\swarrow\searrow))[r]\to P(C(\swarrow\searrow))[r]$ given by the degree $r$ shift of $a\langle C \rangle$.

Then we have $e_{v}M^{r}\cap CM=\{g(b_{n,1_{v,\delta}(\swarrow\searrow)})\mid g\in Ma\}$ where, as in Definition \ref{ppdefsub}, $Ma$ is the set of morphisms of the form $fa$ with $f:P(C(\swarrow\searrow))[r]\to M$ in $\mathcal{K}_{\mathrm{min}}(\Lambda\text{-}\mathbf{Proj})$.
\end{lemma}
We prove Lemma \ref{noice} after Example \ref{examplecompactmorph}.
\begin{example}\label{examplecompactmorph}Consider the gentle algebra $\Lambda$ from Example \ref{runningexample} and the generalised word $[C]=[h^{-1}][g^{-1}][(ft)^{-1}][s^{-1}][r^{-1}][yf][x^{-1}]$ from Example \ref{runningexample2}. Recall 
$[C(\swarrow\searrow)]={}^{\infty}([r][s][t]):[C]([x^{-1}])^{\infty}$ by Example \ref{neswexample}. Similarly $[1_{3,\delta}(\swarrow\searrow)]={}^{\infty}([r][s][t]):([h^{-1}][g^{-1}][f^{-1}])^{\infty}$. The compact morphism $a=a\langle C\rangle$ may be depicted by
\[
\xymatrix@C=.005em@R=1.2em{ 
& & & & & & &\Lambda e_{3}\ar[dl]_{t}\ar[dr]^{h} & & & & & & & & & &  & \\
& & & & & & \Lambda e_{2}\ar[dl]_{s} & & \Lambda e_{4}\ar[dr]^{g} & & & & & & & & & & \\
& & & & & \Lambda e_{5}\ar[dl]_{r} & & \Lambda e_{3}\ar@{=}[uu]_(0.35){a^{0}}\ar[dl]^{t}\ar[dr]_{h} & & \Lambda e_{1}\ar[dr]^{ft} & & & & & & & & & \\
& & & & \Lambda e_{3}\ar[dl]_{t} & & \Lambda e_{2}\ar@{=}[uu]_(0.65){a^{1}}\ar[dl]^{s} & & \Lambda e_{4}\ar@{=}[uu]_(0.35){a^{1}}\ar[dr]_{g} & & \Lambda e_{2}\ar[dr]^{s} & & & & & & & & \\
& & & \Lambda e_{2}\ar[dl]_{s} & & \Lambda e_{5}\ar@{=}[uu]_(0.65){a^{2}}\ar[dl]^{r} & & & & \Lambda e_{1}\ar@{=}[uu]_(0.35){a^{2}}\ar[dr]_{f} & & \Lambda e_{5}\ar[dr]^{r} &  & \Lambda e_{0}\ar[dl]^{yf}\ar[dr]^{x} & & & & & \\
& & \Lambda e_{5}\ar[dl]_{r} & & \Lambda e_{3}\ar[dl]^{t}\ar@{=}[uu]_(0.65){a^{3}} & & & & & & \Lambda e_{3}\ar[dr]_{h}\ar[uu]_(0.35){a^{3}}^{t} & & \Lambda e_{3} & & \Lambda e_{0}\ar[dr]^{x}  & & & & \\
& \Lambda e_{3}\ar[dl] & \iddots & \Lambda e_{2}\ar[dl]^{s}\ar@{=}[uu]_(0.65){a^{4}} & & & & & & & & \Lambda e_{4}\ar[dr]_{g} & & & &  \Lambda e_{3}\ar[dr] & & & \\
\iddots & & \iddots & & & & & & & & & & \ddots & & & & \ddots &  &  
}\]
where the morphisms $a^{n}:P^{n}(1_{3,\delta}(\swarrow\searrow))\to P^{n}(C(\swarrow\searrow))$ are directed vertically upward, and the abscence of an arrow labelled $a^{n}$ indicates that this summand of $P^{n}(1_{3,\delta}(\swarrow\searrow))$ is annihilated by $a$. By Corollary \ref{corgradedspace}, for any object $M$ of $\mathcal{K}_{\mathrm{min}}(\Lambda\text{-}\mathbf{Proj})$ the subspace $CM\subseteq e_{3}M$ is given by the set of $m\in e_{3}M^{0}$ such that there exists a sequence $m_{0},\dots,m_{7}\in M$ satisying the relations listed in Example \ref{runningexample3}.

By definition, $d_{t,M}(m_{0})=tm_{-1}$ for some $m_{-1}\in e_{2}M$,  $d_{s,M}(m_{-1})=tm_{-2}$ for some $m_{-2}\in e_{5}M$, and so on. Similarly $d_{x,M}(m_{7})=xm_{8}$ for some $m_{8}\in e_{0}M$, $d_{x,M}(m_{8})=xm_{9}$ for some $m_{9}\in e_{0}M$, $d_{x,M}(m_{9})=xm_{10}$ for some $m_{10}\in e_{0}M$ and so on. Choosing $g\in Ma$ is equivalent to choosing a $\Lambda$-module homomorphism $f^{n}:P^{n}(C(\swarrow\searrow))\to M^{n}$ for each $n\in\mathbb{Z}$ such that the collection $f=(f^{n}\mid n\in\mathbb{Z})$ is compatible with differentials. To do so extend $f^{n}(b_{i,C(\swarrow\searrow)})=m_{i}$ linearly over $\Lambda$ for each $i\in\mu^{-1}_{C}(n)$. That $f$ is compatible with differentials follows by construction: for example,
\[d^{0}_{M}(m)=tm_{-1}+hm_{1}=f^{1}(d^{0}_{t,P(C(\swarrow\searrow))}(b_{0,C(\swarrow\searrow)})+d^{0}_{h,P(C(\swarrow\searrow))}(b_{0,C(\swarrow\searrow)})).
\]
\end{example}
\begin{proof}[of Lemma \ref{noice}] Without loss of generality we simplify to the case where $r=0$. By Definitions \ref{ppdefsub} and \ref{compactmorph} it suffices to show $CM= X$ where
\[
X=\{f(b_{n,C(\swarrow\searrow)})\mid f:P(C(\swarrow\searrow))\to M\text{ in }\mathcal{C}_{\mathrm{min}}(\Lambda\text{-}\mathbf{Proj})\}.
\] 
Let $m\in e_{v}M^{0}\cap CM$, and so there exists $m_{n},\dots,m_{n+t}\in M$ such that $m=m_{n}$ and $l_{i}m_{i}=r_{i}m_{i+1}$ for all $i$ with $n\leq i<t$. To show $m\in X$ we construct a sequence $(m_{i}\mid i\in I)$ of elements $m_{i}\in M$ such that $b_{i,C(\swarrow\searrow)}\mapsto m_{i}$ defines a morphism $f:P(C(\swarrow\searrow))\to M$. 

Without loss of generality assume $C$ is non-trivial, say $C=l_{1}^{-1}r_{1}\dots l_{t}^{-1}r_{t}$. Let $I_{-}$ and $I_{+}$ be the subsets of $\mathbb{N}$ for which $(C_{\leq n})^{-1}$ is a homotopy $I_{-}$-word and $C_{>n+t}$  is a homotopy $I_{+}$-word. Note that $n-h,n+t+j\in I$ for all $h\in I_{-}$ and all $j\in I_{+}$. We begin by iteratively constructing $m_{n-h}\in M$ for all $h\in I_{-}$ and $m_{n+t+j}\in M$ for all $j\in I_{+}$, noting that $m_{n}$ and $m_{n+t}$ have already been defined. Suppose that $h\in I_{-}$, $m_{n-h}$ has been defined and that $h+1\in I_{-}$. By construction $((C_{\leq n})^{-1})_{h+1}=d_{\gamma}^{-1}\gamma$ for some arrow $\gamma$. Furthermore $\mathrm{im}(d_{\gamma,M})\subseteq \gamma M$, and we choose $m_{n-(h+1)}\in e_{t(\gamma)}M$ such that $d_{\gamma,M}(m_{n-h})=\gamma m_{n-(h+1)}$. 

Similarly if ($j\in I_{+}$, $m_{n+t+j}$ has been defined and $j+1\in I_{+}$) then $(C_{>n+t})_{j+1}=d_{\beta}^{-1}\beta$ for some arrow $\beta$, and we choose $m_{n+t+(j+1)}\in e_{t(\beta)}M$ such that $d_{\beta,M}(m_{n+t+j})=\beta m_{n+t+(j+1)}$. It is straightforward to check that $f(b_{i,C(\swarrow\searrow)})= m_{i}$ ($i\in I$) satisfies $f(b^{+}_{i,C(\swarrow\searrow)}+b^{-}_{i,C(\swarrow\searrow)})= d_{M}(m_{i})$ for all $i$. This is done by separting the cases $i<n$, $i=n$, $n<i<n+t$, $i=n+t$ and $i>n+t$. The cases $i<n$ and $i>n+t$ are similar. As are the cases $i=n$ and $i=n+t$. Note that $f$ is graded of degree $0$ by Corollary \ref{corgradedspace}. This shows $e_{v}M^{0}\cap CM\subseteq X$. The proof that $X\subseteq e_{v}M^{0}\cap  CM$ is similar, easier, and omitted. 
\end{proof}
\section{$\Sigma$-pure-injective complexes and split relations.}\label{splitsection}
\begin{assumption} Throughout \S\ref{splitsection} we let $M$ be an object of $\mathcal{K}_{\mathrm{min}}(\Lambda\text{-}\mathbf{Proj})$, $v$ be a vertex, let $\delta\in\{1,-1\}$ and let $C\in\mathcal{W}_{v,\delta}$ be a homotopy-word.
\end{assumption}
In Definition \ref{definonesidedfunctors} we recall functors, defined more generally in \cite{Ben2018}, of the form $C^{\pm}:\mathcal{C}_{\mathrm{min}}(\Lambda\text{-}\boldsymbol{\mathrm{Proj}})\rightarrow k\text{-}\boldsymbol{\mathrm{Mod}}$ where $C$ is a homotopy $I$-word with $I\subseteq \mathbb{N}$. For this we recall some inclusions in the lattice of subspaces of $e_{h(C)} M$ of the form $CM$.
\begin{remark}By \cite[Corollary 2.1.9]{Ben2018}, if $\alpha\in\mathbf{A}$ then $\alpha^{-1}d_{\alpha}\rad(M)\subseteq e_{t(\alpha)}\rad(M)$, and furthermore given $\beta\in\mathbf{P}$ with $\alpha\beta\in\mathbf{P}$ we
have $(\alpha\beta)^{-1}\alpha d_{\beta}M=\beta{}^{-1}d_{\beta}M$. By \cite[Corollary 2.1.10]{Ben2018} if $\lambda,\eta,\gamma,\sigma,\lambda\eta\in\mathbf{P}$, $h(\gamma)=h(\sigma)$ and $ \mathrm{l}( \gamma)\neq  \mathrm{l}( \sigma)$ then we have the following inclusions
\[
\begin{array}{c}
\begin{array}{ccc}
\eta^{-1}d_{ \mathrm{l}( \eta)}M\subseteq(\lambda\eta)^{-1}d_{ \mathrm{l}( \lambda)}M, &  & d_{ \mathrm{l}( \lambda)}^{-1}\lambda\eta M\subseteq d_{ \mathrm{l}( \lambda)}^{-1}\lambda M,\end{array}\\
\begin{array}{ccccc}
\lambda^{-1}d_{ \mathrm{l}( \lambda)}M\subseteq d_{ \mathrm{l}( \eta)}^{-1}\eta0, &  & \gamma M\subseteq d_{ \mathrm{l}( \sigma)}^{-1}\sigma0, &  & d_{ \mathrm{l}( \sigma)}M\subseteq d_{ \mathrm{l}( \sigma)}^{-1}\sigma0.\end{array}
\end{array}
\]
Let $\alpha,\beta\in\mathbf{A}$ and let $C\alpha^{-1}d_{\alpha}$ and $Cd_{\beta}^{-1}\beta$ be homotopy words. By \cite[Corollary 2.1.15]{Ben2018} we have that: if $\gamma'\in\mathbf{P}$ is longer than $\gamma\in\mathbf{P}$ and $\mathrm{f}(\gamma')=\mathrm{f}( \gamma)=\alpha$ then $C\gamma^{-1}d_{ \mathrm{l}( \gamma)}M\subseteq C\gamma'^{-1}d_{ \mathrm{l}( \gamma')}M$; and if $\tau'\in\mathbf{P}$ is longer than $\tau\in\mathbf{P}$ and  $\mathrm{l}( \tau')=\mathrm{l}( \tau)=\beta$ then $Cd_{\beta}^{-1}\tau'M\subseteq Cd_{\beta}^{-1}\tau M$.
\end{remark}
We now define subfunctors $C^{+}$, $C^{-}$ and ($C$ for when $C$ is finite) of the forgetful functor $\mathcal{C}_{\mathrm{min}}(\Lambda\text{-}\boldsymbol{\mathrm{Proj}})\rightarrow k\text{-}\boldsymbol{\mathrm{Mod}}$ (see Remark \ref{rem6.10}). The inclusions above are used to determine compatibility properties between these functors.
\begin{definition}
\label{definonesidedfunctors}\cite[Definition 2.1.17]{Ben2018} Suppose $C$ is finite. If $a$ is an arrow and $Cd_{a}^{-1}a$ is
a homotopy word let $C^{+}(M)$ be the intersection $\bigcap_{\beta} Cd_{a}^{-1}\beta\rad(M)$ 
over $\beta\in\mathbf{P}$ with
$\mathrm{l}( \beta)=a$. By \cite[Lemma 2.1.19]{Ben2018}, if there are finitely many such $\beta$ then $C^{+}(M)= Cd_{a}^{-1}0$, and otherwise $C^{+}(M)=\bigcap_{\beta} Cd_{a}^{-1}\beta M$. If there is no such arrow $a$ we let $C^{+}(M)=CM$.

If there exists an arrow $b$ where $Cb^{-1}d_{b}$ is a homotopy word let $C^{-}(M)$ be the union $\bigcup_{\alpha} C\alpha^{-1}d_{ b}M$  over all $\alpha\in\mathbf{P}$ with
$\mathrm{f}(\alpha)=b$. Otherwise let $C^{-}(M)=C(\sum d_{c(+)}M+\sum c(-)M)$
where $c(\pm)$ runs through all arrows with head $h(C^{-1})$ and
sign $\pm s(C^{-1})$.

Suppose instead $C$ is a homotopy $\mathbb{N}$-word. Let $C^{+}(M)$ be the set of all $m\in e_{v}M$ with
a sequence of elements $(m_{i})\in\prod_{i\in\mathbb{N}}e_{v_{C}(i)}M$
satisfying $m_{0}=m$ and $m_{i}\in l_{i+1}^{-1}r_{i+1}m_{i+1}$ for
each $i\geq0$, and let $C^{-}(M)$ be the subset of $C^{+}(M)$ where
each sequence $(m_{i})$ is eventually zero.
\end{definition}
\begin{remark}
\label{rem6.10}By \cite[Corollary 6.13]{Ben2016} (see also \cite[Corollary 2.1.20]{Ben2018}) we have that:
\begin{enumerate}
\item the assignments $M\mapsto C^{+}(M)$, $M\mapsto C^{-}(M)$ and ($M\mapsto CM$ for when $C$ is finite) respectively define subfunctors $C^{+}$, $C^{-}$ and ($C$ for when $C$ is finite) of the forgetful functor $\mathcal{C}_{\mathrm{min}}(\Lambda\text{-}\boldsymbol{\mathrm{Proj}})\rightarrow k\text{-}\boldsymbol{\mathrm{Mod}}$ such that $C^{-}\leq C^{+}$  and ($C^{+}\leq C$ when $C$ is finite);
\item if $I_{C}$ is finite then the functor $C$ preserves small coproducts and products; and
\item the functors $C^{\pm}$ preserve small coproducts.
\end{enumerate}
\end{remark}
\begin{corollary}\label{coolbeans}If $M$ is $\Sigma$-pure-injective and $C$ is a homotopy $\mathbb{N}$-word then for each $r\in\mathbb{Z}$ we have $e_{v}M^{r}\cap C^{+}(M)=e_{v}M^{r}\cap C_{\leq l}M$ for some $l\in\mathbb{N}$.
\end{corollary}
\begin{proof}Note that there is a unique integer $n$ such that for all $l\in\mathbb{N}$ the word $B(l)=C_{\leq l}(\swarrow\searrow)$ satisfies $B(l)_{\leq n}
=1_{v,\delta}(\swarrow)$, as in Definition \ref{compactmorph}. Let $X=P(1_{v,\delta}(\swarrow\searrow)[r]$, and for each $l$ let $Y_{l}=P(B(l))[r]$. By Corollary \ref{compactcor}(ii) the objects $X$ and $Y_{l}$ are all compact objects in the triangulated category $\mathcal{K}(\Lambda\text{-}\boldsymbol{\mathrm{Proj}})$. In the notation of Lemma \ref{noice}, for each $l\in\mathbb{N}$ let $a_{l}=a\langle B(l)\rangle[r]$, which is a morphism $X\to Y_{l}$. By Lemma \ref{noice} we have
\[
e_{v}M^{r}\cap C_{\leq l}M=\{g(b_{n,1_{v,\delta}(\swarrow\searrow)})\mid g\in Ma_{l}\}
\]
By Lemma \ref{nicemaps} we have a descending chain
\[
Ma_{1}\supseteq Ma_{2} \supseteq \dots\supseteq Ma_{i} \supseteq \dots
\]
of pp-definable subgroups of $M$ of sort $X$. By assumption the object $M$ of the compactly generated (by Theorem \ref{neemantheorem}) triangulated category $\mathcal{K}(\Lambda\text{-}\boldsymbol{\mathrm{Proj}})$ is $\Sigma$-pure-injective. By the equivalence of (i) and (iii) in Theorem \ref{characterisation}, the above chain stabilises. Altogether this shows there is some $l$ for which $e_{v}M^{r}\cap C_{\leq l}M=e_{v}M^{r}\cap C_{\leq l+i}M$ for all $i\in\mathbb{N}$.
\end{proof}
\begin{definition}\label{Polaroid}
\cite[Definition 2.2.10]{Ben2018} If $B$ and $D$ are homotopy words such that $C=B^{-1}D$ is a $p$-periodic homotopy word then there is a homotopy word $E=l_{1}^{-1}r_{1}\dots l_{p}^{-1}r_{p}$ with
\[ \begin{array}{c}
     B= r_{p}^{-1}l_{p} \dots r_{1}^{-1}l_{1}r_{p}^{-1}l_{p}\dots r_{1}^{-1}l_{1}r_{p}^{-1}l_{p}\dots\text{ and } \\
     D=l_{1}^{-1}r_{1}\dots l_{p}^{-1}r_{p}l_{1}^{-1}r_{1}\dots l_{p}^{-1}r_{p}l_{1}^{-1}r_{1}\dots
\end{array}
\]
In this case we write $D=E^{\,\infty}$, $B=(E^{-1})^{\,\infty}$ and $C={}^{\infty}E{}^{\,\infty}$. For each $n\in\mathbb{Z}$ let 
\[
E_{M}(n)=E\cap(e_{v}M^{n}\oplus e_{v}M^{n})=\{(m,m')\in e_{v}M^{n}\oplus e_{v}M^{n}\mid m'\in Em\}.
\]
Note that $E_{M}(n)$ is a linear relation on $e_{v}M^{n}$, and hence we can consider the object $(e_{v}M^{n},E_{M}(n))$ of the category $k\text{-}\mathbf{Rel}$ from Definition \ref{defrelations}.

By \cite[Lemma 4.5]{Cra2018}  there is a $k$-vector space automorphism of $E_{M}(n)^{\sharp}/E_{M}(n)^{\flat}$ defined by sending $m+E_{M}(n)^{\flat}$ to $m'+E_{M}(n)^{\flat}$
if and only if $m'\in E_{M}(n)^{\sharp}\cap(E_{M}(n)^{\flat}+E_{M}(n)m)$. 
\end{definition}
\begin{lemma}\label{cardylemma}Let $n\in\mathbb{Z}$ and $M$ be an object of $\mathcal{K}_{\mathrm{min}}(\Lambda\text{-}\mathbf{Proj})$, and let $\kappa$ be the cardinality of the structure $\mathsf{M}$ underlying $M$. Let $E=l_{1}^{-1}r_{1}\dots l_{p}^{-1}r_{p}$ be a homotopy word such that ${}^{\infty}E{}^{\,\infty}$ is a ($p$-periodic) homotopy $\mathbb{Z}$-word, as in Definition \ref{Polaroid}. Let $\Gamma$ be the  Kronecker quiver. 

Then the $k\Gamma$-module associated to the object $(e_{v}M^{n}, E_{M}(n))$ of \emph{$k\text{-}\textbf{Rel}$}, as in Definition \ref{defrelations}, has cardinality at most $3\kappa$.
\end{lemma}
\begin{proof}Note that the underlying set of the $k\Gamma$-module in question is $E_{M}(n)\oplus e_{v}M^{n}$, where the action of the arrows in $\Gamma$ is given by the two projections of the left-hand summand into the right-hand one. For the remainder of the proof, if $\mathcal{A}$ is a (locally small) category with objects $X$ and $Y$, let $\mathrm{Hom}_{\mathcal{A}}(X,Y)$ denote the associated set of homomorphisms $X\to Y$ in $\mathcal{A}$. 

Consider the vector space isomorphisms
\[e_{v}M^{r}\simeq\mathrm{Hom}_{\Lambda\text{-}\mathbf{Mod}}(\Lambda e_{v}, M^{r})\simeq \mathrm{Hom}_{\mathcal{K}(\Lambda\text{-}\mathbf{Proj})}(X,Y)\simeq \mathrm{Hom}_{\mathcal{K}(\Lambda\text{-}\mathbf{Proj})}(P(1_{v,\delta})[r],M)
\]
where the complexes $X$ and $Y$ are given by respectively concentrating the projective $\Lambda$-modules $\Lambda e_{v}$ and $M^{r}$ in degree $0$. By Definition \ref{defcardy} this shows $\vert e_{v}M^{r}\vert\leq\kappa$, and since $E_{M}(n)$ is a subset of $e_{v}M^{n}\oplus e_{v}M^{n}$, the claim follows.
\end{proof}
\begin{lemma}\label{cmpctgen}Suppose $M$ is a $\Sigma$-pure-injective object of $\mathcal{K}_{\mathrm{min}}(\Lambda\text{-}\mathbf{Proj})$. Let $E=l_{1}^{-1}r_{1}\dots l_{p}^{-1}r_{p}$ be a homotopy word such that ${}^{\infty}E{}^{\,\infty}$ is a ($p$-periodic) homotopy $\mathbb{Z}$-word, as in Definition \ref{Polaroid}. Then for any $n\in\mathbb{Z}$ the object $(e_{v}M^n, E_{M}(n))$ of \emph{$k\text{-}\textbf{Rel}$} is $\Sigma$-pure-injective.
\end{lemma}

\begin{proof} 
Let $\Gamma$ be the Kronecker quiver. For simplicity we identify any object in $k\text{-}\mathbf{Rel}$ with the corresponding $k\Gamma$-module in $k\Gamma\text{-}\mathbf{Mod}_{\mathrm{rel}}$. By the equivalence of (i) and (v) in \cite[Theorem 8.1]{JenLen1989} it suffices to show that there is a cardinal $\kappa'$ such that the product $(e_{v}M^{n},E_{M}(n))
^{I}$ is a coproduct of $k\Gamma$-modules of cardinality at most $\kappa'$. By Theorem \ref{neemantheorem} and Corollary \ref{sigmacardy} there exists a cardinal $\kappa$ such that the underlying structure of any indecomposable pure-injective object of $\mathcal{K}_{\mathrm{min}}(\Lambda\text{-}\mathbf{Proj})$ has cardinality at most $\kappa$. Let $\kappa'=3\kappa$. Now let $I$ be a set and let $K=M^{I}$. By the equivalence of (i) and (iv) in Theorem \ref{characterisation} there is a set $\mathtt{S}$ such that $K\simeq \bigoplus_{s\in\mathtt{S}} U_{s}$ where each $U_{s}$ is an indecomposable $\Sigma$-pure-injective object of $\mathcal{K}_{\mathrm{min}}(\Lambda\text{-}\mathbf{Proj})$. By Corollary \ref{isosreflect} this means $K\simeq \bigoplus_{s\in\mathtt{S}} U_{s}$ in the category $\mathcal{C}_{\mathrm{min}}(\Lambda\text{-}\mathbf{Proj})$ of complexes. By Remark \ref{rem6.10}(ii) we have that
\[(e_{v}\bigoplus U^{n}_{s},E_{\bigoplus U_{s}}(n))=\bigoplus(e_{v}U^{n}_{s},E_{U_{s}}(n))\text{, and }(e_{v}K^{n},E_{K}(n))=(e_{v}M^{n},E_{M}(n))
^{I}\]
where the coproducts run over $s\in\mathtt{S}$. Since the cardinality of the underlying structure of each $U_{s}$ is at most $\kappa$, the cardinality of each $(e_{v}U^{n}_{s},E_{U_{s}}(n))$ is at most $3\kappa=\kappa'$ by Lemma \ref{cardylemma}. 
\end{proof}

\section{Functorial filtrations.}\label{6}
\begin{definition}\label{defintion.7.12}\cite[Definition 2.2.16]{Ben2016} Let $\Sigma$ be the set of triples $(B,D,n)$ where $n\in\mathbb{Z}$ and such that $(B,D)\in\mathcal{W}_{v,\pm 1}\times\mathcal{W}_{v,\mp 1}$. 
\end{definition}
\begin{assumption}\label{ass9}
In \S\ref{6} fix an
object $M$ of $\mathcal{K}_{\mathrm{min}}(\Lambda\text{-}\boldsymbol{\mathrm{Proj}})$, fix $(B,D,n),(B',D',n')\in\Sigma$ and let $C=B^{-1}D$ and $C'=B'^{-1}D'$.
\end{assumption}
\begin{definition}\label{definition.7.14}\cite[Definition 2.2.14]{Ben2018} We write $C\sim C'$ 
if and only if $C'=C^{\pm1}[t]$ for some $t\in\mathbb{Z}$. So either ($C'=C^{\pm1}$ and $I_{C}\neq\mathbb{Z}\neq I_{C'}$) or ($C'=C^{\pm1}[t]$ and $I_{C}=I_{C'}=\mathbb{Z}$) \cite[Lemma 2.2.17]{Ben2018} (see also \cite[Lemma 2.1]{Cra2018}). Define the \textit{axis} $a_{B,D}\in\mathbb{Z}$ of $(B,D)$ by $C_{\leq a_{B,D}}=B^{-1}$ and $C_{>a_{B,D}}=D$. If $I_{C}=\{0,\dots,m\}$ then $a_{D,B}=m-a_{B,D}$; if $I_{C}=\pm\mathbb{N}$ then $a_{D,B}=-a_{B,D}$; and if $I_{C}=\mathbb{Z}$ then $a_{B,D}=0$ \cite[Lemma 2.2.15]{Ben2018}. 
We write $(B,D,n)\sim(B',D',n')$ if and only if
\[
C\sim C''\text{ and }n'-n=\begin{cases}
\mu_{C}(a_{B',D'})-\mu_{C}(a_{B,D}) & \mbox{(if }C'=C\mbox{ is not a homotopy }\mathbb{Z}\mbox{-word)}\\
\mu_{C}(a_{D',B'})-\mu_{C}(a_{B,D}) & \mbox{(if }C'=C^{-1}\mbox{ is not a homotopy }\mathbb{Z}\mbox{-word)}\\
\mu_{C}(\pm t) & \mbox{(if }C'=C^{\pm1}[t]\mbox{ is a homotopy }\mathbb{Z}\mbox{-word)}
\end{cases}
\]
By \cite[Lemma 2.2.19]{Ben2018} $\sim$ is an equivalence relation. Let $\Sigma(s)$ be the set of $(B,D,n)\in\Sigma$ where $B^{-1}D$ is aperiodic, and $\Sigma(b)$ the set of such $(B,D,n)$ where $B^{-1}D$ is periodic. Note that the relation $\sim$ on $\Sigma$ restricts to an equivalence relation $\sim_{s}$ (respectively $\sim_{b}$) on $\Sigma(s)$ (respectively $\Sigma(b)$). Let $\mathcal{I}(s)\subseteq\Sigma(s)$ (respectively $\mathcal{I}(b)\subseteq\Sigma(b)$) denote a chosen
collection of representatives $(B,D,n)$, one for each
equivalence class of $\Sigma(s)$ (respectively $\Sigma(b)$). Let $\mathcal{I}=\mathcal{I}(s)\sqcup\mathcal{I}(b)$.
\end{definition}
The relation $\sim$ from Definition \ref{definition.7.14} helps us identify and distinguish between isomorphism classes of the functors from Definition \ref{refined}; see Lemma \ref{usefulprops}.
\begin{definition}\label{refined}\cite[Definition 2.2.2]{Ben2018} For any $n\in\mathbb{Z}$ consider the $k$-subspaces of $e_{v}M^{n}$
\[
\begin{array}{c}
F_{B,D,n}^{+}(M)=M^{n}\cap\left(B^{+}(M)\cap D^{+}(M)\right),\\
F_{B,D,n}^{-}(M)=M^{n}\cap\left(B^{+}(M)\cap D^{-}(M)+B^{-}(M)\cap D^{+}(M)\right),\\
G_{B,D,n}^{\pm}(M)=M^{n}\cap\left(B^{-}(M)+D^{\pm}(M)\cap B^{+}(M)\right).
\end{array}
\]
Define the quotients $F_{B,D,n}(M)$ and $G_{B,D,n}(M)$ by
\[
\begin{array}{cc}
F_{B,D,n}(M)=F_{B,D,n}^{+}(M)/F_{B,D,n}^{-}(M), &  G_{B,D,n}(M)=G_{B,D,n}^{+}(M)/G_{B,D,n}^{-}(M).
\end{array}
\]
\end{definition}
\begin{remark}\label{remreffun}If $C$ is aperiodic then $F_{B,D,n}$ and $G_{B,D,n}$ define naturally isomorphic additive functors $\mathcal{K}_{\mathrm{min}}(\Lambda\text{-}\boldsymbol{\mathrm{Proj}})\rightarrow k\text{-}\boldsymbol{\mathrm{Mod}}$ \cite[Corollary 2.2.8]{Ben2018}. If $C$ is periodic then $F_{B,D,n}$ and $G_{B,D,n}$ define
naturally isomorphic functors $\mathcal{K}_{\mathrm{min}}(\Lambda\text{-}\boldsymbol{\mathrm{Proj}})\rightarrow k[T,T^{-1}]\text{-}\boldsymbol{\mathrm{Mod}}$ \cite[Corollary 2.2.12]{Ben2018}. Furthermore, in case $C$ is periodic and we have $C={}^{\infty}E{}^{\,\infty}$ in the notation of Definition \ref{Polaroid}, then by \cite[Lemma 2.2.11]{Ben2018} we have $F^{+}_{B,D,n}(M)= E_{M}(n)^{\sharp}$ and $F^{-}_{B,D,n}(M)= E_{M}(n)^{\flat}$ in the notation of Definition \ref{defsplit}.
\end{remark}
We now gather together some properties of the functors from Definition \ref{refined}. Recall Definition \ref{Polaroid}: if $E$ is a homotopy $\{0,\dots,p\}$ word such that $p>1$ and $C={}^{\infty}E^{\infty}$ is $p$-periodic then the relation $E_{M}(n)$ on $e_{v}M$ is defined to be the set of pairs $(m,m')$ with  $m'\in Em$.
\begin{lemma}\label{usefulprops}Let $\mathrm{res}$ denote the involution of $k[T,T^{-1}]\text{-}\boldsymbol{\mathrm{Mod}}$ which swaps the action of $T$ and $T^{-1}$. Fix $(B,D,n)$, $(B',D',n')$, $C$ and $C'$ as in Assumption \ref{ass9}.
\begin{enumerate}
\item \emph{\cite[Corollary 2.2.24]{Ben2018}} If $(B,D,b)\sim(B',D',n')$ then the following statements hold.
\begin{enumerate}
\item  If $C$ is aperiodic then $F_{B,D,n}\simeq F_{B',D',n'}$.
\item  If $C$ is periodic and $C'=C[t]$ for some $t\in\mathbb{Z}$ then $F_{B,D,n}\simeq F_{B',D',n'}$. 
\item  If $C$ is periodic and $C'=C^{-1}[t]$ for some $t\in\mathbb{Z}$ then $F_{B,D,n}\simeq\mathrm{res} \,F_{B',D',n'}$.
\end{enumerate}
\item Let $C$ be aperiodic and $P=P(C)[\mu_{C}(a_{B,D})-n]$. 
\begin{enumerate}
\item\emph{\cite[Lemma 2.3.20]{Ben2018}} If $C'=C$ and $n-n'=\mu_{C}(a_{B,D})-\mu_{C}(a_{B',D'})$ then $F_{B',D',n'}(P)\simeq k$ as vector spaces. If $(B,D,n)\nsim(B',D',n')$ then $F_{B',D',n'}(P)=0$.
\item\emph{\cite[Lemma 2.5.2]{Ben2018}} If $\mathcal{B}$
is a $k$-basis of $F_{B,D,n}(M)$ then there is a morphism $\theta_{B,D,n,M}:P^{(\mathcal{B})}\rightarrow M$
where $F_{B,D,n}(\theta_{B,D,n,M})$ is an isomorphism.
\end{enumerate}
\item Let $C={}^{\infty}E{}^{\,\infty}$ be $p$-periodic, $V$ be a $k[T,T^{-1}]$-module and $P=P(C,V)[-n]$. 
\begin{enumerate}
\item\emph{\cite[Lemma 2.3.21]{Ben2018}} If $C'=C[m]$ and $n-n'=\mu_{C}(m)$ then $F_{B',D',n'}(P)\simeq V$ as $k[T,T^{-1}]$-modules. If $(B,D,n)\nsim(B',D',n')$ then $F_{B',D',n'}(P)=0$.
\item\emph{\cite[Lemma 2.5.4]{Ben2018}} If the object $(e_{v}M^{n},E_{M}(n))$ of \emph{$k\text{-}\mathbf{Rel}$} is split then there is a morphism $\theta_{B,D,n,M}:P\rightarrow M$ where $F_{B,D,n}(\theta_{B,D,n,M})$ is an isomorphism.
\end{enumerate}
\end{enumerate}
\end{lemma}
Note that the statements of \cite[Lemma 2.5.4]{Ben2018} and Lemma \ref{usefulprops}(iiib) are equivalent: any relation $(V,C)$ that \textit{admits a reduction which meets in} $0$ must have been split by \cite[Corollary 1.4.33]{Ben2018}.
\section{Covering properties.}\label{seccover}
In \S\ref{seccover} we gather some consequences of what has been developed so far. These consequences will be seen to be vital to the proof of Theorem \ref{characterisation}.
\begin{assumption}Throughout \S\ref{seccover} fix an
object $M$ of $\mathcal{K}_{\mathrm{min}}(\Lambda\text{-}\boldsymbol{\mathrm{Proj}})$.
\end{assumption}
\begin{lemma}
\label{lemma.5.3}\emph{\cite[Lemma 2.4.8]{Ben2018} (}see also \emph{\cite[Lemma 10.3]{Cra2018})}. Fix an integer
$r$ and some $\delta\in\{\pm1\}$. For any non-empty subset $S$
of $e_{v}M^{r}$ which does not meet \emph{$\rad(M)$} there
is a homotopy word $C\in\mathcal{W}_{v,\delta}$ such that either:
\begin{enumerate}
\item $C$ is finite and $S$ meets $C^{+}(M)$ but not $C^{-}(M)$; or
\item $C$ is a homotopy $\mathbb{N}$-word and $S$ meets $C_{\leq n}M$
but not $C_{\leq n}\rad(M)$ for each $n\geq0$.
\end{enumerate}
\end{lemma}
\begin{lemma}
\label{lemma.5.2-1}\emph{(}See \emph{\cite[Lemma 2.4.1]{Ben2018}} and \emph{\cite[Lemma 10.5]{Cra2018})}. Let $M$ be $\Sigma$-pure-injective, $r\in\mathbb{Z}$, $\delta=\pm1$ and $U$ be an
$k$-subspace of $e_{v}M^{r}$ with $e_{v}\mathrm{rad}(M^{r})\subseteq U$. If $m\in e_{v}M^{r}\setminus U$ then
there exists $B\in\mathcal{W}_{v,\delta}$ and $D\in\mathcal{W}_{v,-\delta}$
such that $U+m$ meets $G_{B,D,r}^{+}(M)$ but not $G_{B,D,r}^{-}(M)$. 
\end{lemma}
\begin{proof}
By adapting the argument in the proof of \cite[Lemma 10.5]{Cra2018} (with few complications), it suffices to show that there is a homotopy word $C\in\mathcal{W}_{v,\delta}$
such that $U+m$ meets $C^{+}(M)$ but not $C^{-}(M)$. Let $S=U+m$. Note $S\cap\rad(M)=\emptyset$ since $e_{v}\rad(M^{r})\subseteq U$ and $m\notin U$. So by Lemma \ref{lemma.5.3} there is a homotopy word $C$
such that either $C$ is finite and $S\cap C^{+}(M)\neq\emptyset=S\cap C^{-}(M)$,
or $C$ is a homotopy $\mathbb{N}$-word and for all $n\geq0$ we
have $S\cap C_{\leq n}M\neq\emptyset=S\cap C_{\leq n}\rad(M)$.
We may assume $I_{C}=\mathbb{N}$, and so, by Corollary \ref{coolbeans}, for some $l>0$ we have
\[S\cap C^{+}(M)=S\cap e_{v}M^{r}\cap C^{+}(M)=S\cap e_{v}M^{r}\cap C_{\leq l}M\neq\emptyset
\]
as required.
\end{proof}
Recall, from Definition \ref{definition.7.14}, that: $\Sigma(s)$ (respectively $\Sigma(b)$) is the set of triples $(B,D,n)\in\Sigma$ where $B^{-1}D$ is aperiodic  (respectively periodic); $\mathcal{I}(s)$ (respectively $\mathcal{I}(b)$) is a collection of representatives $(B,D,n)$; and $\mathcal{I}=\mathcal{I}(s)\sqcup\mathcal{I}(b)$.
\begin{lemma}
\label{lemma.7.5}\emph{(}See \emph{\cite[Lemma 2.5.5]{Ben2018}}, \emph{\cite[Lemma 10.5]{Cra2018}} and \emph{\cite[p. 163]{ButRin1987})}. Let $\theta:P\rightarrow M$ be a morphism in \emph{$\mathcal{K}_{\mathrm{min}}(\Lambda\text{-}\boldsymbol{\mathrm{Proj}})$} where $M$ is $\Sigma$-pure-injective. If $F_{B,D,n}(\theta)$ is surjective for each $(B,D,n)\in\Sigma$ then $\theta^{i}$
is surjective for each $i$.
\end{lemma}
\begin{proof}
For a contradiction suppose that $\theta^{i}$ is not surjective for
some $i\in\mathbb{Z}$. Since $\Lambda$ is perfect $M^{i}$ is a projective cover of $M/\mathrm{rad}(M^{i})$, and so $\mbox{rad}(M^{i})$ is a superfluous
submodule of $M^{i}$. This means $e_{v}\mbox{im}(\theta^{i})+e_{v}\mbox{rad}(M^{i})$ is contained in a maximal $k$-subspace $U$ of $e_{v}M^{i}$. Since
$e_{v}\mbox{rad}(M^{i})\subseteq U$ and $U\neq e_{v}M^{i}$, by Lemma \ref{lemma.5.2-1}(ii)
for some element $m\in e_{v}M^{i}\setminus U$ there are homotopy words $B\in\mathcal{W}_{v,\delta}$
and $D\in\mathcal{W}_{v,-\delta}$ for which ($B^{-1}D$ is a homotopy word)
and $U+m$ meets $G_{B,D,i}^{+}(M)$ but not $G_{B,D,i}^{-}(M)$. From here one can show $F_{B,D,i}(\theta)$ is not surjective by adapting the argument from the proof of \cite[Lemma 10.6]{Cra2018}.
\end{proof}

\begin{lemma}\emph{\cite[Lemma 2.5.6]{Ben2018} (}see also \emph{\cite[Lemma 9.4]{Cra2018})}.\label{lemma.7.1-1} Let $N$ be a coproduct of shifts of string and band complexes. Let $\theta:N\rightarrow M$ be a morphism in $\mathcal{K}_{\mathrm{min}}(\Lambda\text{-}\boldsymbol{\mathrm{Proj}})$
where $\bar{F}_{B,D,n}(\theta)$ is injective for all $(B,D,n)\in\mathcal{I}$.
Then each $\theta^{i}$ is injective.
\end{lemma}
\section{\label{sec:Proofs-of-the}Completing the proof of the main theorem.}
We now combine what we have found so far. The proof of Theorem \ref{maincor} below shows how the functorial filtrations method works.
\begin{proof}[of Theorem \ref{maincor}]
Let $M$ be a $\Sigma$-pure-injective object of $\mathcal{K}_{\mathrm{min}}(\Lambda\text{-}\boldsymbol{\mathrm{Proj}})$. Recall, from Definition \ref{definition.7.14}, the equivalence relation on the triples $(B,D,n)$ where $B^{-1}D$ is a homotopy word and $n\in\mathbb{Z}$. Recall that $\mathcal{I}=\mathcal{I}(s)\sqcup \mathcal{I}(b)$ where $\mathcal{I}(s)$ (respectively $\mathcal{I}(b)$) denotes a chosen set of representatives $(B,D,n)$ such that $B^{-1}D$ is aperiodic (respectively periodic). Recall that if $(B,D,n)$ lies in $\mathcal{I}(s)$ (respectively $\mathcal{I}(b)$) then the functor $F_{B,D,n}$ has the form $\mathcal{K}_{\mathrm{min}}(\Lambda\text{-}\mathbf{Proj})\rightarrow k\text{-}\mathbf{Mod}$ (respectively $\mathcal{K}_{\mathrm{min}}(\Lambda\text{-}\mathbf{Proj})\rightarrow k[T,T^{-1}]\text{-}\mathbf{Mod}$).

By Lemma \ref{usefulprops}(iib), if $(B,D,n)\in\mathcal{I}(s)$ then for any basis $\mathcal{B}$ of $F_{B,D,n}(M)$ there is a morphism $\theta^{s}_{B,D,n,M}:\bigoplus P(C)[\mu_{C}(a_{B,D})-n]\to M$ where (the coproduct runs through $\mathcal{B}$ and) $F_{B,D,n}(\theta_{B,D,n,M})$ is an isomorphism. Now suppose $(B,D,n)\in\mathcal{I}(b)$ with $B^{-1}D={}^{\infty}E{}^{\,\infty}$ in the notation of Definition \ref{Polaroid}. By Lemma \ref{cmpctgen} the object $(e_{v}M^n, E_{M}(n))$ of $k\text{-}\textbf{Rel}$ is $\Sigma$-pure-injective. By Corollary \ref{meandbill} (and Remark \ref{remreffun}) this means $(e_{v}M^n, E_{M}(n))$ is split and $V=F_{B,D,n}(M)$ is a $\Sigma$-pure-injective $k[T,T^{-1}]$-module. By Lemma \ref{usefulprops}(iic) this means there is a morphism $\theta^{b}_{B,D,n,M}:P(C,V)[-n]\to M$ where $F_{B,D,n}(\theta_{B,D,n,M})$ is an isomorphism. 

By the universal property of the coproduct, the morphisms $\theta^{s}_{B,D,n,M}$ and $\theta^{b}_{B,D,n,M}$ together define a morphism $\theta:N\to M$ where $N$ is a coproduct of shifts of string complexes and shifts of band complexes which are paramenterised by $\Sigma$-pure-injective $k[T,T^{-1}]$-modules. Furthermore, by construction we have that $F_{B,D,n}(\theta)$ is an isomorphism for each $(B,D,n)\in\mathcal{I}$. By Lemma \ref{usefulprops}(i), (iia) and (iiia) this means that $F_{B,D,n}(\theta)$ is an isomorphism for each $(B,D,n)\in\Sigma$. 

Since $N$ is a coproduct of string and band complexes, by Lemma \ref{lemma.7.1-1} this means $\theta^{i}$ is injective for each $i\in\mathbb{Z}$. Similarly, since $M$ is a $\Sigma$-pure-injective object of $\mathcal{K}_{\mathrm{min}}(\Lambda\text{-}\boldsymbol{\mathrm{Proj}})$, by Lemma \ref{lemma.7.5} $\theta^{i}$ is surjective for each $i\in\mathbb{Z}$. 
\end{proof}
\begin{conjecture}The author conjectures that Theorem \ref{maincor} can be generalised by extending the scope of the classification, from $\Sigma$-pure-injective objects to all pure-injectives. 
\end{conjecture}
\begin{remark}\label{lastremark}
Recall that (shifts of) string and band complexes are indecomposable objects in $\mathcal{K}_{\mathrm{min}}(\Lambda\text{-}\boldsymbol{\mathrm{Proj}})$ by Theorem \ref{theorem.1.1}(ii). Hence, by Corollary \ref{decompcorr}, if two coproducts of shifts of string and band complexes are both isomorphic and $\Sigma$-pure-injective, then there is an isoclass preserving  bijection between the summands.
\end{remark}
By combining Theorems \ref{maincor} and \ref{theorem.1.2} with Remark \ref{lastremark}, one can determine exactly when two $\Sigma$-pure-injective objects in the homotopy category are isomorphic. Note that any indecomposable $\Sigma$-pure-injective $k[T,T^{-1}]$-module is isomorphic to: an indecomposable finite-dimensional module;
a \textit{Pr\"ufer} module (an injective envelope of a simple); or the function field $k(T)$. We now note how the decomposition of any $\Sigma$-pure-injective object of $\mathcal{K}_{\mathrm{min}}(\Lambda\text{-}\boldsymbol{\mathrm{Proj}})$, as in Theorem \ref{maincor}, is unique up to reordering and isomorphism. 
\begin{theorem}
\label{theorem.1.2}\emph{\cite[Theorem 1.2]{Ben2016}} Let $C$ and $E$ be homotopy words, let $V$ and $W$ be $k[T,T^{-1}]$-modules  and let $n\in\mathbb{Z}$. 
\begin{enumerate}
\item If $C$ and $E$ are aperiodic, then $P(C)[n]\simeq P(E)$ in $\mathcal{K}(\Lambda\text{-}\boldsymbol{\mathrm{Proj}})$ if and only if:
\begin{enumerate}
\item we have  $I_{C}=\{0,\dots,m\}$ and $(I_{E},E,n)=(I_{C},C,0)\text{ or } (I_{C},C^{-1},\mu_{C}(m))$; or
\item we have $I_{C}=\pm\mathbb{N}$ and $(I_{E},E,n)=(\pm\mathbb{N},C,0)\text{ or }(\mp\mathbb{N},C^{-1},0)$; or
\item we have $I_{C}=\mathbb{Z}$ and  $(I_{E},E,n)=(\mathbb{Z}, C^{\pm1}[t],\mu_{C}(\pm t))$ for some $t\in\mathbb{Z}$.
\end{enumerate}
\item If $C$ and $E$ are periodic, then $P(C,V)[n]\simeq P(E,W)$ in $\mathcal{K}(\Lambda\text{-}\boldsymbol{\mathrm{Proj}})$ if and only if:
\begin{enumerate}
\item we have $E=C[t]$, $V\simeq W$ and $n=\mu_{C}(t)$ for some $t\in\mathbb{Z}$; or
\item we have $E=C^{-1}[t]$, $V\simeq \mathrm{res}\,W$ and $n=\mu_{C}(-t)$ for some $t\in\mathbb{Z}$.
\end{enumerate}
\item If $C$ is aperiodic and $E$ is periodic, then $P(C)[n]\not\simeq P(E,V)$ in $\mathcal{K}(\Lambda\text{-}\boldsymbol{\mathrm{Proj}})$.
\end{enumerate}
\end{theorem}
\bibliography{biblio}
\bibliographystyle{plain}

\end{document}